\documentclass[a4paper,12pt,reqno]{article}
 \usepackage[english]{babel}
  \usepackage{amsmath,amsfonts,amssymb,amsthm,bbm}
   \usepackage{graphics,epsfig,psfrag} 
   \usepackage{paralist,scrextend,subfigure}
   \usepackage[dvipsnames]{xcolor}
   \usepackage{stmaryrd} 	
   \usepackage{hyperref} 
\usepackage{float}

\usepackage{esint} 

    \usepackage[active]{srcltx} 
    \usepackage{color,soul} 
    \usepackage[latin1]{inputenc}
    \usepackage[OT1]{fontenc}
    \usepackage{authblk}

\newtheorem{theorem}{Theorem}[section]
\newtheorem{corollary}[theorem]{Corollary}
\newtheorem{lemma}[theorem]{Lemma}
\newtheorem{proposition}[theorem]{Proposition}

\theoremstyle{definition}
\newtheorem{definition}[theorem]{Definition}
\newtheorem{remark}{Remark}

\newtheorem{hypothesis}{Hypothesis}

\DeclareMathOperator\diam{diam}

\DeclareMathOperator\supp{supp}

\DeclareMathOperator*{\esssup}{ess\,sup}
\DeclareMathOperator*{\essinf}{ess{\,}inf}

\def\N{\mathbb{N}}
\def\Z{\mathbb{Z}}

\def\R{\mathbb{R}}

\let\O=\Omega
\let\e=\varepsilon

\let\t=\tilde
\let\ol=\overline
\let\ul=\underline
\let\.=\cdot
\let\0=\emptyset

\let\mc=\mathcal

\def\1{\mathbbm{1}}

\def\Sph{S^{N-1}}

\newcommand{\su}[2]{\genfrac{}{}{0pt}{}{#1}{#2}}

\def\seq#1{(#1_n)_{n\in\N}}
\def\limn{\lim_{n\to\infty}}

\newenvironment{formula}[1]{\begin{equation}\label{#1}}
                       {\end{equation}\noindent}

\def\Fi#1{\begin{formula}{#1}}
\def\Ff{\end{formula}\noindent}

\title{\bf Blocking and invasion for reaction-diffusion equations in periodic media}

\author[*]{Romain {\sc Ducasse}}
\author[*]{Luca {\sc Rossi}}
\affil[*]{Ecole des Hautes Etudes en Sciences Sociales, PSL Research University,  
		Centre d'Analyse et Math\'ematiques Sociales, 54 boulevard Raspail 
		75006 Paris, France}

\date{}

\begin{document}

\maketitle


\noindent {\textbf{Keywords:} }  reaction-diffusion equations, invasion, propagation, spreading, domains with holes.
\\
\noindent {\textbf{MSC:} } 35A08, 35B30, 35K05, 35K57, 35B40

\begin{abstract}
We investigate the large time behavior of solutions of 
reaction-diffusion equations 
with general reaction terms in periodic media. 
We first derive some conditions which guarantee that solutions with
compactly supported initial data invade the domain.
In particular, we relate such solutions with front-like solutions
such as pulsating traveling fronts. Next, 
we focus on the homogeneous equation set in a domain with periodic holes, 
and specifically in the cases where
fronts are not known to exist. We show how the geometry of the domain can block or allow invasion. We finally exhibit a periodic domain on which the propagation 
takes place in an asymmetric fashion, in the sense that the invasion 
occurs in a direction but is blocked in the opposite one.
\end{abstract}

\section{Introduction}

\subsection{Large time behavior for the Cauchy problem}\label{intro}

Reaction-diffusion equations classically arise in the study of biological phenomena 
(propagation of genes, epidemics), 
in physics (combustion) and more recently in social sciences (rioting models).
They have been extensively studied since the seminal papers of Fisher \cite{Fisher} and
Kolmogorov, Petrovski and Piskunov \cite{KPP} 
who dealt with the {\em homogeneous} equation
\begin{equation}\label{FKPP}
\partial_t u=\Delta u+f(u), \quad  t>0,  \ x\in\mathbb{R}^{N}.
\end{equation}
A crucial progress in the study of \eqref{FKPP} is due to Aronson and Weinberger \cite{AW}.
The basic assumption there is that $f(0)=f(1)=0$. 
Then the authors consider three different sets of hypotheses.
With the terminology commonly employed in the literature, they are:
\begin{labeling}{\quad monostable\quad}
  \item[\quad {\em monostable}\quad] 
  $f>0\quad\text{in }(0,1)$;
  
  \item[\quad {\em combustion}] 
  $\exists\theta\in(0,1),\quad 
f=0\quad\text{in }[0,\theta],\quad f>0\quad\text{in }(\theta,1)$;

  \item[\quad {\em bistable}]  
  $\exists\theta\in(0,1),\quad 
f<0\quad\text{in }(0,\theta),\quad f>0\quad\text{in }(\theta,1)$.
\end{labeling}

The monostable case is the one considered in \cite{Fisher}, \cite{KPP} and includes the logistic equation $f(u) = u(1-u)$; the prototype of the bistable term is 
$f(u) = u(1-u)(u-\theta)$, which reduces to the Allen-Cahn nonlinearity when $\theta=1/2$.

Two key features of the equation \eqref{FKPP} are exhibited in \cite{KPP,AW}. First, 
the existence of a special type of solutions named \emph{traveling fronts} (or \emph{traveling waves}). 
These are solutions of the form $u(t,x) = \phi(x\cdot e -c t)$, with $e$ on the unit sphere
$\mathbb{S}^{N-1}$, 
$c \in \mathbb{R}$ and $\phi(z) \to 1$ as $z$ goes to $-\infty$ and $\phi(z)\to 0$ as $z$ goes to $+\infty$.
The second feature is the \emph{invasion} property:
for any ``large enough'' compactly supported non-negative initial datum, the solution of
\eqref{FKPP} converges locally uniformly to $1$ as time goes to infinity. 
This result requires the hypothesis $\int_0^1 f>0$, which is automatically fulfilled
in the monostable and combustion cases.
How large the initial datum needs to be depends on the type of nonlinearity: 
in the monostable case it is sufficient to be larger than a positive constant on a large ball,
in the combustion or bistable case the constant needs to be larger than $\theta$.
Actually, in the monostable case, there exists a critical exponent $\beta>1$
such that if $f(u)\geq u^\beta$ for $u\sim0$ (hence
in particular if $f'(0)>0$) then all solutions with non-negative, not identically equal to $0$
initial data converge to $1$ as $t$ goes to $+\infty$. 
This is known as the \emph{hair-trigger effect}, c.f.~\cite{AW}.
Observe instead that in the combustion case, if the initial datum lies below $\theta$,
then \eqref{FKPP} reduces to the heat equation and thus
the solution converges uniformly to $0$ as $t$ goes to $+\infty$. 
In the bistable case the situation is even worse; moreover if 
$\int_{0}^{1}f=0$ then all solutions with compactly supported initial 
data $\leq1$ stay bounded away from $1$, and they all 
converge to $0$ if $\int_{0}^{1}f<0$.

The aim of this paper is to investigate the question of invasion 
for reaction-diffusion equations in periodic media.
Specifically, we consider the problem
\Fi{evol}
\begin{cases}
	\partial_t u=\nabla\cdot (A(x) \nabla u)+q(x)\cdot \nabla u+f(x,u), & t>0,\ x\in\O\\
	\nu \cdot A(x)\nabla u=0, & t>0,\ x\in\partial\O.
\end{cases}
\Ff
combined with a nonnegative initial datum $u_0$ 
\footnote{ Initial data and solutions are always understood to be bounded in order to avoid non-uniqueness issues.}.
Here $\Omega\subset\mathbb{R}^{N}$ is an unbounded smooth domain which is periodic in 
the directions of the canonical basis. Typical examples are domains with ``holes" which are 
periodically arranged, that is,
$\Omega = \left( K + L\mathbb{Z}^{N} \right)^{c}$
(here and in the sequel, for a given subset $A$ of $\R^N$, we denote
its complement by $A^{c}:=\R^N\setminus A$)
where $K\subset \mathbb{R}^{N}$ is a compact set and $L>0$.
In the case $\O=\R^N$, we neglect the second equation in \eqref{evol}.
The diffusion matrix $A$, the drift term $q$ and the nonlinearity $f$ are also assumed to be 
periodic with respect to the $x$ variable, with the same period
as $\O$.

We shall describe the large time behavior of solutions in terms of the following properties.

	\begin{definition}\label{def}
	We say that a solution $u$ satisfies the properties of
	\begin{labeling}{\quad persistence \ }
		\item[\quad\em blocking ] if \ $\displaystyle \limsup_{t \to +\infty} 
		\left(\sup_{x \in \O} u(t,x) \right)<1 $;
		\item[\quad\em persistence] if \ $\displaystyle\liminf_{t\to+\infty}\Big(\min_{x\in K}u(t,x)\Big)>0$ for all compact $K\subset\ol\O$;
		\item[\quad\em invasion] if  \ $u(t,x)\to1$ as $t$ goes to $+\infty$, 
		locally uniformly in $x\in\ol\O$.
	\end{labeling}
	\end{definition}
Persistence is compatible with both blocking and invasion, whereas the latter two properties are mutually exclusive.

We shall say that a function $u : \O \mapsto [0,1]$ is \emph{front-like}
if
\begin{equation}\label{def f l}
\exists e\in \mathbb{S}^{N-1},\quad
\lim_{\su{x\cdot e \to -\infty}{x\in\O}}u(x) = 1,\quad
\lim_{\su{x\cdot e \to +\infty}{x\in\O}}u(x) = 0.
\end{equation}
The kind of questions we address here are:
\begin{itemize}

\item Does the validity of the invasion property for {\em all}\, front-like data 
implies that the invasion property holds for {\em some}\, compactly supported data ?

\item Does the validity of the invasion property for a {\em single}\, front-like datum 
implies that the invasion property holds for {\em some}\, compactly supported data ?

\end{itemize}

We will show in Theorem \ref{invasion general} that the answer to the first question is yes, specifying also how ``large" the compactly supported data need to be. 
Then we will provide a negative answer to the second question, see Remark \ref{rem:single} after Theorem \ref{th oriented}.

\subsection{Statement of the main results}\label{sec:results}

Throughout the paper, 
$\O$ is a domain with $C^{3}$ 
boundary, $\nu:\partial\O\to \mathbb{S}^{N-1}$ denotes its exterior normal field and $\partial_\nu$
the associated directional derivative. We further assume $\O$ to be 
periodic, i.e., there are $L_1,\dots,L_N>0$ such that
$$
\forall k \in L_1\Z\times\cdots\times L_N\Z,\quad
\Omega +\{ k\} = \Omega.
$$
We let $\mc{C}$ denote the periodicity cell: 
$$\mathcal{C} := \ol\O\cap \prod_{i = 1}^{N}[0,L_{i}).$$
We assume, unless otherwise stated, that the coefficients are also periodic with respect to $x$ with the same periodicity as $\O$, i.e., 
\[
\forall k \in L_1\Z\times\cdots\times L_N\Z,\ x\in\ol\O,\quad
A(x+k)=A(x), \quad q(x+k) = q(x),
\]
\[
\forall k \in L_1\Z\times\cdots\times L_N\Z,\ x\in\ol\O,\ s\in[0,1],\quad
f(x+k,s)=f(x,s).
\]
The following regularity and ellipticity hypotheses will
always be understood in the paper: $A \in C^{3}(\overline{\Omega})$ is a symmetric, 
uniformly elliptic matrix field, $q \in C^{1,\alpha}(\overline{\Omega})$, for some $\alpha \in (0,1)$ is a vector field and $f : \ol\Omega \times [0,1] \mapsto \mathbb{R}$ is globally Lipschitz-continuous. We shall denote $0<\lambda\leq \Lambda$ the ellipticity constants of the matrix field $A$, so that
\begin{equation}\label{A}
\forall x \in \ol\O, \quad   \lambda I_{N} \leq A(x) \leq \Lambda I_{N},
\end{equation}
where $I_{N}$ is the identity matrix and the order is the usual one on symmetric matrices.
We shall further assume that $f$ satisfies the following properties:
$$\forall x \in \ol\Omega, \quad f(x,0) = f(x,1) =0, $$
$$\exists S \in (0,1) \ \text{such that} \ s\mapsto f(x,s)  \text{ is strictly positive for all}\ x \in \ol\O,\ s\in (S,1).$$
%
%
This means that $0$ and $1$ are steady states for \eqref{evol}, with $1$ being stable whereas $0$ could be both
stable or unstable. 
We extend $f(x,s)$ by
a negative function for $s\not\in [0,1]$ such that the resulting function is globally Lipschitz-continuous.
The above properties allow us to define
\begin{equation}\label{deftheta}
\theta := \max \left\{ s \in [0,1)  \ :\  \exists x\in\ol{\O},\ 
	f(x,s)=0    \right\}.
\end{equation}
Namely, $\theta\in[0,1)$ is the smallest value for which $f>0$ in $\ol{\O}\times(\theta,1)$.

This paper deals with the phenomena of invasion, persistence and blocking - formulated in
 Definition \ref{def} -
 for equation \eqref{evol}, and is divided into two main parts. In the first part, we give sufficient conditions for persistence and invasion to occur for ``large enough" compactly supported initial data. 
 In the second part, we study how the geometry of the domain $\Omega$ can influence the invasion in the bistable case.

\subsubsection{Persistence and invasion}\label{invasion results}

Our first result provides a sufficient
condition for the persistence of solutions to \eqref{evol} with ``large enough" initial data, 
in the absence of the drift term~$q$.

\begin{theorem}\label{th persistence}
Assume that $q\equiv0$ and that $f$ satisfies 
%
\begin{equation}\label{mean positive}
 \int_{\mathcal{C}}\int_{0}^{1}f(x,s)\,ds\, dx>0.
\end{equation}
Then, for all $\eta\in(\theta,1)$, where $\theta$ is defined by \eqref{deftheta}, there is $r>0$ such that the persistence property holds for every solution to \eqref{evol} arising from a
non-negative initial datum  satisfying
	$$u_0>\eta\quad \text{ in }\ \O\cap B_r.$$
\end{theorem}
Let us make one comment about the condition $q\equiv0$. 
It comes from the fact that our proof of Theorem \ref{th persistence} relies on 
an energy method. This condition is presumably not optimal, but it guarantees some control on the drift: indeed, a drift that is ``too strong" would lead to extinction for all compactly supported initial data, no matter how large they are.


Our second main result 
is the equivalence of the invasion property for initial data which are ``large enough
on a large set'' and for front-like initial data, i.e., satisfying~\eqref{def f l}. 

\begin{theorem}\label{invasion general}
The following properties are equivalent:
\begin{enumerate}[$(i)$]
\item Invasion occurs for all solutions of \eqref{evol} with non-negative front-like initial data.

\item Invasion occurs for ``large enough" initial data, in the sense that for all $\eta\in(\theta,1)$, where $\theta$ is defined by \eqref{deftheta}, there is $r>0$ such that the invasion property holds for every solution to \eqref{evol} arising from a
non-negative initial datum  satisfying
	$$u_0>\eta\quad \text{ in }\ \O\cap B_r.$$

\end{enumerate}
\end{theorem}
The fact that $(ii)$ implies $(i)$ in Theorem \ref{invasion general} is an immediate
consequence of the parabolic comparison principle. The interest is in the the other implication. 
We mention that a related result is contained in the work \cite{W} by Weinberger.
However, the hypotheses there are more restrictive than our condition $(i)$ in Theorem~\ref{invasion general}:
first, compared with our definition~\eqref{def f l}, front-like data in \cite{W} are assumed
to be strictly smaller than $1$, second, the emerging solutions
need to have a positive spreading speed rather than just satisfy the invasion property. 
The method of \cite{W} relies on a discrete dynamical system approach. In particular, the rather involved argument that allows the author to handle compactly supported initial data cannot be directly performed in the continuous PDE setting.

As already mentioned, it may be complicated to check that condition $(i)$ of Theorem \ref{invasion general} holds. Our next result -- Theorem \ref{inv ptf} --
provides a sufficient condition for the invasion property expressed in terms 
of \emph{pulsating traveling fronts}.  
%
%
These are particular entire solutions of \eqref{evol} satisfying some structural properties that generalize to the periodic framework the notion of traveling front. They were first introduced in dimension $N=1$ by Shigesada, Kawasaki and Teramoto in \cite{SKT}.
\begin{definition}\label{def:ptf}
A {\em pulsating traveling front} 
in the direction $e \in \mathbb{S}^{N-1}$ with speed $c \in \mathbb{R}\backslash\{0 \}$ connecting $1$ to $0$ is an entire (i.e., for all times) solution of \eqref{evol} $0<v<1$ such that
	\Fi{ptf}
	\begin{cases}
		\forall z\in \displaystyle\prod_{i=1}^{N}L_{i}\Z,\ x\in\O,\quad v(t+\frac{z\.e}{c},x)=v(t,x-z)\\
		v(t,x)\underset{x\cdot e \to -\infty}{\longrightarrow} 1 \ \text{ and } \
		v(t,x)\underset{x\cdot e \to +\infty}{\longrightarrow}  0 .
	\end{cases}
	\Ff
	A pulsating traveling front with speed $c=0$ is a front-like stationary solution of~\eqref{evol},
	in the sense of \eqref{def f l}.
\end{definition}

Observe that, if $v$ is a pulsating traveling front, then $x \mapsto v(t,x)$ is front-like for all $t\in\R$. 
The use of traveling fronts to study the large time behavior of solutions
of the Cauchy problem in the combustion and bistable cases is quite a natural approach, already used in the pioneering paper \cite{AW} for the homogeneous equation \eqref{FKPP}. For later use, we introduce the following:

\begin{hypothesis}\label{H2}
For every direction $e \in \mathbb{S}^{N-1}$, there is a pulsating traveling front with
speed $c>0$.
\end{hypothesis}
The requirement that the speeds of the fronts are positive in Hypothesis \ref{H2} is 
necessary for the invasion property to hold, 
in the sense that, owing to the parabolic comparison principle, the existence of -- even a single -- pulsating traveling front with speed $c\leq 0$ prevents the invasion of all solutions with compactly supported initial data smaller than $1$. 
\begin{theorem}\label{inv ptf}
Assume that there is $\delta>0$ such that
$$s\mapsto f(x,s) \text{ is nonincreasing on 
	 $(0,\delta)$ and on $(1-\delta,1)$}.$$	
Then, Hypothesis \ref{H2} implies that properties 
$(i)$-$(ii)$ of Theorem \ref{invasion general} hold.	
\end{theorem}

The usefulness of this result lies in the fact that there is
a huge literature about pulsating traveling fronts. In particular, their existence 
and the positivity of their speed 
is proved by Berestycki and Hamel in \cite[Theorems 1.13 , 1.14]{pulsating}
when $f$ is of the monostable or combustion type (cf.~the definition given in
Section \ref{intro} in the homogeneous case)
under the following additional assumptions on~$q$:
\begin{equation}\label{type1 q}
\begin{cases}
\nabla \cdot q =0 \quad & \text{in }\ol\O, \\
\int_{ \mathcal{C}}q = 0, \\
q\cdot \nu = 0 \quad & \text{on } \partial \Omega.
\end{cases}
\end{equation}
Observe that in the monostable case
the monotonicity hypothesis of Theorem \ref{inv ptf} is
not fulfilled, but since the larger the nonlinearity,
	the more likely the invasion property holds, the idea is to consider 
	a term $\tilde f\leq f$ of combustion type.
 This allows to derive the  following result 
concerning non-negative reaction terms, such as monostable and combustion.
\begin{corollary}\label{invasion combustion}
Assume that $q$ satisfies \eqref{type1 q} and that
$f$ satisfies 
\begin{equation}\label{combustion}
\forall x \in \Omega, \ \forall s \in (0,1), \quad f(x,s)\geq 0.
\end{equation}
Then, properties $(i)$-$(ii)$ of Theorem \ref{invasion general} hold.
\end{corollary}
We point out that Corollary \ref{invasion combustion} could be also derived by combining
the results of~\cite{pulsating} with those of Weinberger \cite{W} mentioned before.

Under the generality of our assumptions on $f$, which include the bistable case,
the question of the existence of fronts is still widely open, in particular if the domain $\Omega$ is not $\mathbb{R}^{N}$. We are not aware of any result in such case. 
If the equation is set on $\mathbb{R}^{N}$, then Xin \cite{Xin1,Xin2} and Ducrot \cite{Ducrot} derive the existence of fronts with positive speed under some conditions on the coefficients. Theorem \ref{inv ptf} then yields the invasion result in such cases, 
see Corollaries \ref{invasion bistable}, \ref{invasion bistable 2} below.

One may wonder whether Hypothesis \ref{H2} is necessary in Theorem \ref{inv ptf}. This is not the case. Indeed Zlato\v{s} exhibits in \cite{Z1} a bistable periodic nonlinearity such that the $1$-dimensional equation 
$$
\partial_{t}u = \partial_{xx} u +f(x,u), \quad  t> 0, \ x\in \mathbb{R},
$$
admits no pulsating traveling fronts (actually, he shows that there are no \emph{transition fronts}, a notion generalizing pulsating traveling fronts), but such that the invasion property holds for sufficiently large compactly supported initial data. On the other hand, we exhibit in Theorem \ref{th blocking bistable} below a situation where there are no fronts and 
invasion does not occur. We mention that other examples of reaction-diffusion equations that do not admit fronts are known in dimension $1$, see \cite{Xin3}, and in cylindrical domains (with a drift-term), see \cite{BHcylinders}.

Once we know that invasion occurs for a solution with compactly supported initial datum
smaller than $1$, it is not hard to see that this happens with a strictly positive speed, see Remark \ref{vitesse positive} below. In all the situations where Hypothesis \ref{H2} holds, we can actually characterize this speed through an analogous formula to the one derived by Freidlin and G\"artner in \cite{FG,Freidlin} in the KPP case.
This was first proved by the second author in \cite{F-G} when $\Omega = \mathbb{R}^{N}$ and then by the first author in \cite{D} for periodic domains. We mention that Weinberger also derives a similar formula in \cite{W}.

Our next result applies to equations with non-negatives nonlinearities
-- such as monostable or combustion -- set in the whole 
space. It provides an explicit sufficient condition for invasion to occur, as well
as a lower bound on the asymptotic speed at which this takes place. 
This is expressed in terms of the following quantity:
$$
R(f) :=\sup_{0<K<H<1}\bigg((H-K)\min_{\substack{x\in \ol\O\\s \in [K,H]}}f(x,s)\bigg),
$$
which is the area of the largest rectangle one can fit under the graph of $s \mapsto \inf_{x\in \ol\O}f(x,s)$ in the upper half-plane. 
For this result, it is convenient to consider equations in non-divergence form, that is,
\begin{equation}\label{evol2}
\partial_t u= Tr (A(x) D^{2} u)+q(x)\cdot \nabla u + f(x,u), \quad \forall t>0,\ \forall x\in \mathbb{R}^{N},
\end{equation}
under the usual standing assumption of Section \ref{sec:results} on 
the terms $A$, $q$, $f$.
\begin{theorem}\label{invasion2}
Assume that $f$ satisfies \eqref{combustion} and that $q$ is continuous 
and satisfies
$$
\limsup_{\vert x \vert \to +\infty} \, q(x)\cdot\frac{x}{\vert x\vert}<  \frac{\lambda}{\sqrt{\Lambda}} \sqrt{R(f)},
$$
where $\lambda, \Lambda$ are given by \eqref{A}. Then, properties $(i)$-$(ii)$ of 
Theorem \ref{invasion general} hold for the equation \eqref{evol2}.

Moreover, calling $w^{\star} := \frac{\lambda}{\sqrt{\Lambda}} \sqrt{R(f)}-\limsup_{\vert x \vert \to +\infty} q(x)\cdot\frac{x}{\vert x \vert} $, there holds
$$
\forall c< w^{\star}, \quad  \inf_{\vert x \vert \leq ct} u(t,x) \underset{t \to +\infty}{\longrightarrow} 1. 
$$

\end{theorem}
Let us state make some comments about this result. 
First, we point out that it holds without any assumption on $q$ besides boundedness,
unlike Corollary \ref{invasion combustion} which requires \eqref{type1 q}.
Next, the theorem holds true without the periodicity assumption, provided \eqref{combustion} is fulfilled by the function $\inf_{x \in \mathbb{R}^{N}}f(x,\cdot)$. 
This is a non-degeneracy hypothesis, without which the result could not hold (if $f(x,\cdot) \equiv 0$ for $x$ large for instance, there is no way invasion could occur). The proof will actually use only this non-degeneracy hypothesis, which is guaranteed by the periodicity.
Finally, one can observe that the more 
negative $q(x)\cdot \frac{x}{\vert x \vert}$ is, the larger $w^{\star}$ becomes. This is a bit counter-intuitive because $q(x)\cdot \frac{x}{\vert x \vert}<0$ roughly
means that the drift ``gathers the mass" instead of scattering it. Hence, such a drift should slow down the invasion. This suggests that our bound $w^{\star}$ on the invasion speed may not be optimal in general. On the other hand, because this drift prevents the mass to scatter, it is natural that the more negative $q(x)\cdot \frac{x}{\vert x \vert}$ is, the more likely invasion should occur.
In any case, though probably not optimal, our estimate on the invasion speed has the
advantage of being explicit and easy to compute. 
While an estimate in the same spirit is derived in \cite{BHRossi} under the assumption
$\partial_s f(\.,0)>0$, 
we are not aware  of analogous results of this type
in the combustion case.

The proof of Theorem \ref{th persistence} is carried out in Section \ref{sec:persistence}. Theorem~\ref{invasion general} and Theorem~\ref{inv ptf} are derived in Section~\ref{section invasion}.
In Section \ref{sec:cor}, we derive Corollary \ref{invasion combustion}. Theorem~\ref{invasion2} is proved in Section \ref{estimate}.

\subsubsection{Influence of the geometry of the domain}

The results presented above provide some sufficient conditions ensuring the persistence or invasion properties.
It is known that, if $f$ is of the bistable type, 
the geometry of the domain can produce ``obstacles'' which may prevent propagation.
This is observed by Berestycki, Hamel and Matano in \cite{BHM} for an exterior domain (i.e., the complement of a compact set)
and by Berestycki, Bouhours and Chapuisat  in~\cite{BBC}
for a cylindrical-type domain with a bottleneck.
If such obstacles repeat periodically in the domain, one could expect that 
the blocking property holds.
We will show that this is indeed the case, but also that other scenarios are possible.

We consider the simplest problem set in a periodic domain:
\Fi{evol homogeneous}
\begin{cases}
\partial_t u=\Delta u+f(u), & t>0,\ x\in\O \\
\partial_{\nu} u=0, & t>0,\ x\in\partial\O,
\end{cases}
\Ff
where $f$ is an \emph{unbalanced bistable} nonlinearity, i.e., 
\Fi{unbalanced}
\exists\theta\in(0,1),\quad f<0\quad\text{in }(0,\theta),\quad f>0\quad\text{in }(\theta,1), \quad \int_0^1f(s)ds>0.
\Ff
Notice that, under the assumption 
\eqref{unbalanced}, 
Theorem~\ref{th persistence} ensures that the persistence property 
for ``large enough" initial data holds for problem \eqref{evol homogeneous}, 
whatever the domain $\O$ is.
However, we will see that the geometry of $\O$ 
%
affects the way this persistence takes place, in three radically different ways.
Namely, we construct three periodic domains $\O_1$, $\O_2$, $\O_3$ that exhibit, respectively,
invasion, blocking (in the sense of Definition \ref{def}) and a new phenomenon that
we call
{\em oriented invasion}. 
Let us mention that our results could be carried out for the more general equation~\eqref{evol}, 
but we have chosen to emphasize here the role of the geometry of the domain. 

The  domain $\O_1$ is given by the whole space with a star-shaped hole $K$ repeated periodically, with a period $L$ sufficiently large.

\begin{theorem}\label{th invasion bistable}
Let $f$ satisfy \eqref{unbalanced} 
and let $K\subset \mathbb{R}^{N}$ be a star-shaped compact set. Then, for $L>0$ sufficiently large, properties $(i)$-$(ii)$ of Theorem \ref{invasion general} hold for the problem \eqref{evol homogeneous} set on 
$$
\Omega_{1} := \left(K + L\mathbb{Z}^{N}\right)^{c}.
$$
\end{theorem}
%
%
%

%

Next, building on the result of \cite{BHM},
we exhibit a domain $\Omega_{2}$ where the propagation is always~blocked.

\begin{theorem}\label{th blocking bistable}
Let $f$ satisfy \eqref{unbalanced}. 
There exists a periodic domain $\Omega_{2}$ such that, for  
the problem \eqref{evol homogeneous} set on $\Omega_{2}$, the following hold:\begin{enumerate}[$(i)$]

\item Any solution arising from a compactly supported initial datum
$u_{0}\leq 1$ is blocked in the sense of Definition \ref{def}.

\item There exists a solution with compactly supported initial datum which converges 
 (increasingly) to a periodic non-constant stationary solution 
 as $t$ goes to $+\infty$.

\item Invasion fails for every front-like initial datum and 
moreover there exist no pulsating traveling fronts.
\end{enumerate}
\end{theorem}
The failure of the invasion property  is in strong contrast with the homogeneous case $\O = \R^{N}$, where 
 such property is guaranteed by the
 condition $\int_{0}^{1} f(s)ds >0$, at least for large enough initial data.
We point out that 
Theorem~\ref{th blocking bistable} part $(ii)$ implies the existence of
an intermediate periodic steady state between $0$ and~$1$ 
which is {\em stable from below}.
So, it turns out that the geometry of $\O_2$ alters the bistable character of the 
nonlinearity~$f$
making some non-trivial stable steady states appear.
This is the reason why there are no fronts in such case, but one should rather 
expect the existence of {\em propagating terraces} instead, see \cite{Ducrot,GR}.

We finally construct a domain $\Omega_{3}$ which exhibits a new phenomenon that 
we call \emph{oriented invasion}. Namely, invasion occurs in a direction 
$e$, with 
a positive linear speed, whereas the propagation is blocked in the opposite direction $-e$.
We state and prove the theorem in $\mathbb{R}^{2}$, but it can  be easily generalized to higher dimensions. 
We  let $(e_1, e_2)$ denote the unit vectors of the canonical basis of $\mathbb{R}^{2}$.

\begin{theorem}\label{th oriented}
Let $f$ satisfy \eqref{unbalanced}. 
There exists a periodic domain $\Omega_{3}\subset \mathbb{R}^{2}$ and a positive constant $c$ such that, for every $\eta>\theta$, there is $r>0$ for which
the following properties hold for every solution to \eqref{evol homogeneous} arising from a
compactly supported initial datum satisfying 
$$0\leq u_0\leq 1,\qquad
u_0>\eta\quad \text{ in }\ \O\cap B_r.$$
\begin{enumerate}[$(i)$]
\item\emph{Invasion in the direction $e_1$:}
$$\forall 0<c_1<c_2<c,\ \forall a>0,\quad
	\min_{\substack{x\in\ol\O_{3} \\ c_1 t\leq x\cdot e_{1}\leq c_2 t\\ \vert x\cdot e_{2} \vert \leq a}}
	|u(t,x)-1|  \underset{t \to +\infty}{\longrightarrow} 0.$$

\item\emph{Blocking in the direction $-e_1$:}
$$
u(t,x) 
\underset{x\cdot e_{1} \to -\infty}{\longrightarrow} 0\quad
\text{uniformly in }t>0.
$$
\end{enumerate}
\end{theorem}

Theorem~\ref{th oriented} provides an example of a periodic domain on which
invasion takes place in an \emph{asymmetric} way. 
In the KPP case, it turns out that the speed of invasion
in a given direction and in the opposite one do coincide. 
Our result provides a counter-example to this fact in the bistable case.

\begin{remark}\label{rem:single}
From statement $(i)$ of Theorem \ref{th oriented}, it is easy to deduce that the invasion property is verified for any initial datum which is front-like in the direction $e_1$, i.e., satisfying~\eqref{def f l} with $e=e_1$.
This shows that the invasion property for a single front-like initial datum 
does not suffice to guarantee that invasion occurs for ``large enough" compactly supported data. 
Namely, in Theorem \ref{invasion general}, one cannot weaken property $(i)$ by the 
existence of one single front-like initial datum for which invasion occurs.
The possible extension of the theorem would be by assuming that 
in any direction there is a front-like initial datum for which  invasion occurs.
We leave it as an open question.
\end{remark}

The proofs of Theorems \ref{th invasion bistable} and \ref{th blocking bistable}
are given in Sections \ref{invasion} and \ref{blocking} respectively.
Section \ref{Oriented invasion} is dedicated to the proof of Theorem \ref{th oriented}.


\section{Persistence}\label{sec:persistence}
This section is devoted to the proof of the persistence result, Theorem \ref{th persistence}. In this whole section, we assume that $q\equiv 0$ and that $f$ satisfies \eqref{mean positive}. The proof relies on the study of the stationary problem

\Fi{stationary}
\begin{cases}
\nabla \cdot (A(x) \nabla u)+  f(x,u)=0, & x\in\O\\
\nu \cdot A(x)\nabla u=0, & x\in\partial\O.
\end{cases}
\Ff
The main tool is the construction of a family of solutions in truncated domains. This will be achieved using an energy method, in the same spirit of Berestycki, Lions \cite{BL}, where the authors study the existence of positive solutions for homogeneous bistable equations in the whole space $\mathbb{R}^{N}$.

We consider the primitive of $f$, defined by  
$$
F(x,s) := \int_{0}^{s} f(x,\sigma)d\sigma.
$$
For $r>0$, we introduce 
the energy functional $\mc{E}_{r}$ associated with \eqref{stationary} in the truncated domain
$\O\cap B_r$:
$$\mc{E}_r(\phi):=\int_{\O\cap B_r}\left(\frac12 (A(x)\nabla\phi, \nabla \phi)-F(x,\phi)\right),$$
acting on the space $H^1_0(\ol\O\cap B_r)$.
We study the existence of minimizers for this energy. In order to do so, we first derive a geometrical lemma ensuring that 
$\O$ and $B_r$ are not tangent for a.e.~$r>0$. We recall that $\nu(x)$ stands for the exterior normal derivative at point $x\in \partial \O$.
\begin{lemma}\label{lem:nottangent}
For a.e.~$r>0$, there holds 
\Fi{nonparallel}
\forall x\in (\partial\O)\cap\partial B_r,\quad
\nu(x)\neq\pm\frac{x}{|x|}.
\Ff
\end{lemma}
\begin{proof}
Let $d$ be a regularised signed distance from $\partial\O$, that is, 
a smooth function on $\R^N$ coinciding with the signed distance from 
$\partial\O$ in a neighbourhood of $\partial\O$, positive inside $\O$.
Consider the pair of functions
$$\rho_\pm(x):=|x|\pm d(x).$$
This functions are smooth outside the origin. It follows from the Morse-Sard
theorem~\cite{Morse-Sard} that the inverse images $\rho_\pm^{-1}(r)$ do not 
contain 
any critical points of $\rho_\pm$, except for $r$ belonging respectively
to some sets $\mc{R}_{\pm}$ with zero Lebesgue measure.
Hence, for $r\in\R^+\setminus(\mc{R}_-\cup\mc{R}_+)$, that is, for a.e.~$r>0$,
any $x\in \partial\O \cap \partial B_r$ satisfies
$\rho_\pm(x)=|x|=r$ and thus, for such $x$, we have
$$0\neq\nabla\rho_\pm(x)=\frac{x}{|x|}\mp\nu(x).$$
\end{proof}

For all $r$ for which \eqref{nonparallel} holds, the set $\O\cap B_r$ satisfies an interior and exterior cone condition. Also, it is not hard to check that, for all $r$ for which \eqref{nonparallel} holds, the set $\O\cap B_r$
has a finite number of connected components

\begin{proposition}\label{pro:minimiser}
For all $r>0$ such that \eqref{nonparallel} holds, the functional $\mc{E}_r$ admits a global minimiser $\ul u_r \in H^{1}_{0}(\overline{\O} \cap B_{r})$ such that $0 \leq \ul u_{r} \leq 1$ a.e. in $\Omega\cap B_{r}$.
\end{proposition}

\begin{proof}
First, observe that, because we have extended $f(x,s)$ by a negative function for $s>1$, we have
$$
\forall x\in \ol\O, \ \forall s \geq 0, \quad F(x,s) \leq \int_{0}^{\min\{ s, 1\}}f(x,\sigma)d\sigma \leq  \max_{\overline{\Omega} \times [0,1]}f.
$$
We deduce that
$$\mathcal{E}_{r}\geq-  \max_{\overline{\Omega} \times [0,1]}f |B_r|,$$
that is, $\mathcal{E}_{r}$ is bounded from below. Consider a minimizing sequence $(u_{n})_{n\in \mathbb{N}}$ of elements of $H_{0}^{1}(\ol\O\cap B_{r})$. 
We can assume without loss of generality
that $\esssup u_{n} \leq 1$, because, defining $v_{n}=\min(u_{n},1)$, we have that $v_{n}\in H^{1}_0(\ol \O \cap B_{r})$ and $\mathcal{E}_{r}(v_{n}) \leq \mathcal{E}_{r}(u_{n})$ 
again because $f(x,s)<0$ for $s>1$. Likewise, we can take $\essinf u_{n} \geq 0$.
Let us check that the sequence  $\seq{u}$ is bounded in $H^1(\O\cap B_r) $. Indeed, 
on the one hand, the sequence is bounded in $L^2(\O\cap B_r) $,
and on the other hand, remembering that $\lambda$ denotes the ellipticity constant of $A$,
$$\frac{1}{2}\lambda \|\nabla u_n\|^{2}_{L^{2}(\O \cap B_{r})}  - \max_{\ol\O \times [0,1]} f\vert B_{r} \vert    \leq
\mathcal{E}_{r}(u_{n}).$$
Now, we cannot directly apply the Rellich-Kondrachov theorem to the sequence $(u_{n})_{n\in \mathbb{N}}$ because $\Omega \cap B_{r}$ is not necessarily
smooth. However, it is a Lipschitz domain thanks to \eqref{nonparallel}, hence we can apply the usual Sobolev extension theorem (see  \cite[Section 6]{S} or \cite{C}) for $(u_{n})_{n\in \mathbb{N}}$ and then apply the Rellich-Kondrachov theorem to the sequence of extended functions, getting then that, up to extraction, $(u_{n})_{n\in \mathbb{N}}$ converges in the $L^{2}$ norm (up to subsequences) to some $u\in L^{2}(\O\cap B_{r})$.

Let us show that the convergence actually holds in $H_{0}^{1}(\ol \O\cap B_{r})$. 
To do so, we show that $\seq{u}$ is a Cauchy sequence in this space. For $m,n\in \mathbb{N}$, we have
\begin{equation*}
\begin{array}{rl}
\inf_{H^{1}_{0}(\ol\O \cap B_{r})} \mc{E}_{r} &\leq \mc{E}_{r}\Big( \frac{u_{m}+u_{n}}{2} \Big) \\
&\leq \frac{1}{2}\mc{E}_{r}(u_{m})+\frac{1}{2}\mc{E}_{r}(u_{n})-\frac{\lambda}{4}\| \nabla (u_{n}-u_{m}) \|^{2}_{L^{2}( \O \cap B_{r})} \\ 
&\quad + \frac{1}{2}\int_{\O \cap B_{r}}
\Big(F(x,u_{m})+F(x,u_{n})-2F\Big( x, \frac{u_{m}+u_{n}}{2} \Big)\Big).
\end{array}
\end{equation*}
Using the fact that $s \mapsto F(x,s)$ is Lipschitz-continuous, uniformly in $x$, and that $\seq{u}$ is a Cauchy sequence in $L^2(\O\cap B_r)$, and therefore in $L^1(\O\cap B_r)$, we
see that the above integral goes to $0$ as $m,n$ go to $+\infty$. On the other hand,
$\frac{1}{2}\mc{E}_{r}(u_{m})+\frac{1}{2}\mc{E}_{r}(u_{m})\to\inf_{H^{1}_{0}(\ol\O \cap B_{r})} \mc{E}_{r}$ as $m,n$ go to $+\infty$. It follows that
$\seq{\nabla u}$ is a Cauchy sequence in $L^{2}(\O \cap B_{r})$.
Now, by continuity of $\mc{E}_{r}$ in $H^{1}_{0}(\ol\O\cap B_r)$,
we conclude that the limit $u$ is a global minimiser
for~$\mc{E}_{r}$. Finall, the fact that $0\leq u \leq 1$ a.e. follows from the same argument as before.
\end{proof}
 We know that $\ul u_r$ satisfies the Euler-Lagrange equation associated with $\mc{E}_r$, 
together with mixed boundary conditions. Namely, it is a solution of the problem 
\eqref{stationary} inside $B_r$ and vanishes on $\O\cap\partial B_r$ in the sense of the trace.
However, at this stage, we cannot exclude that the minimizer $\ul u_{r}$ is the trivial solution identically equal to zero. Owing to Lemma \ref{lem:nottangent}, we can infer that $\ul u_r$ is continuous on $\ol{\O\cap B_r}$, as shown in the following.

\begin{lemma}\label{lem:ur}
	If \eqref{nonparallel} holds then $\ul u_r\in C^0(\ol{\O\cap B_r})$
	and $\max_{\ol{\O\cap B_r}}\ul u_r<1$.	
\end{lemma}

\begin{proof}
	Let $r$ be such that \eqref{nonparallel} holds. Then, the  weak bounded solution $\ul u_{r}$ is actually continuous up to the boundary of $\Omega\cap B_{r}$, see \cite[Theorem 14.5]{Sta}.
Because $\ul u_r$ is continuous and vanishes on $\O\cap\partial B_r$, it attains
its maximum~$M \leq 1$ at some $\bar x\in\ol\O\cap B_r$. 
Assume by contradiction that $M= 1$. 
The function $v:=M-\ul u_r$ is non-negative, vanishes at $\bar x$
and satisfies $\nabla (A(x) \nabla v )=f(x,M-v)$ in $\O\cap B_r$.
Because $s \mapsto f(x,s)$ is Lipschitz-continuous uniformly in $x$,
 we see that
$\nabla \cdot (A(x) \nabla v)=f(x,M-v)$ can be rewritten as a linear equation with bounded coefficients.
It follows from Hopf's lemma and $\nu A \nabla v=0$ on $(\partial\O)\cap B_r$ that 
$\bar x\notin\partial\O$. Thus,
the strong maximum principle implies that $v\equiv0$ on the connected component 
$\mc{O}$ of $\O\cap B_r$ containing $\bar x$. Observe that $\ol{\mc{O}}\cap\partial B_r
\neq\emptyset$ because $\O$ is connected. As a consequence, since $v$ is continuous up to the boundary, there exists $x\in \Omega \cap \partial B_r$
such that $v(x)=0$, i.e., $\ul u_{r}(x)=M$. We have reached a contradiction because $\underline{u}_{r}=0$ on $\Omega \cap \partial B_{r}$.
\end{proof}

The next lemma states that $\ul u_r \not\equiv 0$, provided $r>0$ is large enough. That is, we have built non-trivial solutions of \eqref{stationary} set on truncated domains.

\begin{lemma}\label{ur not 0}
Assume that $f$ satisfies \eqref{mean positive}. Then, there is $R^{\star}$ such that, for $r>R^{\star}$, there holds
$$
\ul u_r \not\equiv 0.
$$
\end{lemma}

\begin{proof}
	Since  $\ul u_r$ minimises $\mc{E}_r$, we can get an upper bound for
	$\mc{E}_r(\ul u_r)$ by estimating $\mc{E}_r(\phi_{r})$
	on a suitable function $\phi_{r}\in H^1_0(\ol\O\cap B_r)$.
	For $r > 1$, we define $\phi_{r}:\ol\O\cap B_r \to\R$ as follows:
	$$\phi_{r}(x)=\begin{cases}
	1 & \text{if }|x|\leq r-1,\\
	 r-\vert x \vert &\text{if }r-1 <|x|< r.
	\end{cases}$$
	Observe that $\vert \nabla \phi_{r} \vert \leq 1$.	
	We compute
	\[\begin{split}
	\mc{E}_r(\phi_{r}) &=
	\int_{\O\cap B_r}\left(\frac12(A\nabla\phi_{r})\.\nabla\phi_{r}-F(x,\phi_{r})\right)\\
	&\leq\frac12 \Lambda |B_r\setminus B_{r-1}|
	-\int_{\O \cap B_{r-1}}F(x,1)
	-\left(\min_{\ol\O\times[0,1]}F\right)\vert   B_r \setminus B_{r-1} \vert ,
	\end{split}\]
	where $\Lambda$ is given by \eqref{A}. We eventually infer the existence of a constant $C$ independent of $r$ such that
	$$
	\forall  r>1,\quad
	\mc{E}_r(\ul u_r)\leq C r^{N-1}-
	\int_{\O \cap B_{r-1}}F(x,1).
	$$
Observing that, for any measurable periodic function $g$ in $L^{1}_{loc}(\mathbb{R}^{N})$, we have $\fint_{B_{r}} g \to \fint_{[0,L_{1}]\times \ldots \times[0,L_{n}]} g$	as $r$ goes to $+\infty$, where $\fint$ stands for the integral average, we have
$$
|\O\cap B_r|\sim_{r \to +\infty}\frac{\vert\mc{C}\vert}{\prod L_{i}}|B_r|=\frac{\vert\mc{C}\vert}{\prod L_{i}}|B_1|r^N,
$$
and then 
$$
\frac{1}{r^{N}}\int_{\O \cap B_{r-1}}F(x,1)\underset{r \to +\infty}{\longrightarrow} \frac{\vert\mc{C}\vert}{\prod L_{i}}|B_1| \fint_{\mathcal{C}}F(x,1).
$$	
Then, because the latter term is positive by hypothesis \eqref{mean positive}, we have
	
$$
	\limsup_{r \to +\infty}\frac{\mc{E}_r(\ul u_r)}{r^{N}}\leq \limsup_{r \to +\infty} \left( C \frac{1}{r}-
	\frac{1}{r^{N}}\int_{\O \cap B_{r-1}}F(x,1)\right)<0,
$$
whence $\mc{E}_r(\ul u_r) <0$ if $r>0$ is sufficiently large. Therefore, $\ul u_{r}\not\equiv 0$ because $\mc{E}_r(0) =0$, and the result follows.
\end{proof}
Now, we can prove Theorem \ref{th persistence}.

\begin{proof}[Proof of Theorem \ref{th persistence}]
Thanks to Lemma \ref{ur not 0}, we can take $r>0$ such that $\ul u_r \not\equiv 0$. This function, extended by $0$ on $\Omega\backslash \overline{B}_{r}$, is a generalized subsolution of \eqref{evol}. Let $\underline{u}$ denote the solution of \eqref{evol} arising from such initial datum. Using the parabolic comparison principle, it is classical to get that $\underline{u}(t,x)$ is increasing with respect to $t$ and converges locally uniformly in $\overline{\Omega}$ to a stationary solution of \eqref{evol} as $t$ goes to $+\infty$. This stationary solution is strictly positive, thanks to the elliptic strong maximum principle and Hopf principle, and then $\underline{u}(t,x)$ satisfies the persistence property. By the parabolic comparison principle, we can infer that every solution of \eqref{evol} with initial datum $u_{0}$ satisfying $u_{0} \geq \ul u_r$ satisfies the persistence property.

Next, take $\eta \in (\theta , 1)$, and, for $n \in \mathbb{N}$, let $u^{n}_{0}$ be a function with compact support in $\ol\O \cap B_{n}$ such that $u^{n}_{0}\leq \eta$ and
$$
u_{0}^{n} = \eta \quad \text{ in } \  \ol\O \cap B_{n-1}.
$$
Then, we have that $\nu \cdot ( A(x) \nabla u^{n}_{0}) = 0, \ \forall x \in \partial \Omega \cap B_{n-1}$,
i.e., $u_{0}^{n}$ satisfies there the boundary conditions of \eqref{evol}. This is necessary to have the usual parabolic estimates up to time $t=0$, see, for instance \cite[Theorems 5.2, 5.3]{La}. Let $u^{n}(t,x)$ denote the solution of \eqref{evol} arising from the initial datum $u^{n}_{0}$. By the parabolic estimates, $u^{n}$ converges locally uniformly in $[0,+\infty) \times \ol\O$ to the solution $v$ of \eqref{evol} with constant initial datum $v(0,\cdot) \equiv \eta$. Because $\eta >\theta$, $v(t,x)$ converges uniformly to $1$ as $t$ goes to $+\infty$~: indeed, we can define $z(t)$ to be the solution of the ODE $\dot{z}(t) = \min_{x\in \ol\O}f(x,z(t))$ with initial value $z(0) =\eta$. 
This is a subsolution of \eqref{evol}, whence the parabolic comparison principle yields 
$z(t) \leq v(t,x)$ for all $t\geq 0$, $x \in \Omega$. 
Observe that $z(t)$ goes to $1$ as $t$ goes to $+\infty$, because $z(0)>\theta$ and $\theta$ is defined by~\eqref{deftheta} 
as the largest $s\in [0,1)$ such that $f(x,s)$ vanishes. Combining this with the fact that $v \leq 1$, we obtain the uniform convergence of $v(t,\cdot)$ to $1$ as $t$ goes to $+\infty$. Because $\ul u_{r}<1$ and is compactly supported in $\overline{\Omega \cap B_{r}}$, there exists $T>0$ such that
$v(T,\.)>\ul u_r$. We can then find $\bar n$ so that
$$
\forall x \in \overline{\Omega\cap B_{r}},\quad
u^{\bar n}(T,x) \geq \ul u_{r}(x).
$$
The parabolic comparison principle implies that $u^{\bar n}(T+t,\cdot) \geq\underline{u}(t,\cdot)$ and therefore $u^{\bar n}$ satisfies the persistence property. By comparison, the same holds true for any solution of \eqref{evol} with initial datum $u_{0}$ larger than $u^{\bar n}_{0}$, and in particular if $u_{0} >\eta$ on $\ol\O \cap B_{\bar n}$, hence the result.
\end{proof}


Now that Theorem \ref{th persistence} is proved, and before turning to the proof of Theorem \ref{invasion general}, we 
show that, under an additional assumption on $f$, 
we have ``almost invasion".


\begin{proposition}\label{pro:urto1}
Assume that $f$ satisfies \eqref{mean positive} and that

\begin{equation}\label{strict}
\forall x\in \ol\O,\ \forall s\in[0,1),\quad F(x,s)<F(x,1).
\end{equation}
Then, for every compact $K\subset\ol\O$, for every $\varepsilon>0$ and $\eta >\theta$, where $\theta$ is defined by \eqref{deftheta}, there is $r>0$ such that any solution to \eqref{evol} with 
	a non-negative initial datum $u_0$ satisfying
	$$u_0>\eta\quad \text{ in }\ \O\cap B_r$$
satisfies
$$
\liminf_{t\to+\infty} \left( \min_{x\in K} u(t,x) \right) > 1-\varepsilon.
$$
\end{proposition}

\begin{proof}
	The proof is divided into five steps.
	
	\medskip
	{\em Step 1. Estimate on $\mc{E}_r(\ul u_r)$}.\\
	Recalling that  $\ul u_r$ minimises $\mc{E}_r$, we have that 
	$\mc{E}_r(\ul u_r) \leq \mc{E}_r(\phi)$, for any $\phi \in H^1_0(\ol\O\cap B_r)$. As in the proof of Lemma \ref{ur not 0}, for $r>1$, we define $\phi_{r}:\ol\O\cap B_r \to\R$ as follows:
	$$\phi_{r}(x)=\begin{cases}
	1 & \text{if }|x|\leq r-\sqrt{r},\\
	\sqrt r-\frac{|x|}{\sqrt r} &\text{if }r-\sqrt r<|x|< r.
	\end{cases}$$
	We have $\vert \nabla \phi_{r} \vert \leq \frac{1}{\sqrt{r}}$, and then
	\[\begin{split}
	\mc{E}_r(\phi_{r}) &=
	\int_{\O\cap B_r}\left(\frac12(A\nabla\phi_{r})\.\nabla\phi-F(x,\phi_{r})\right)\\
	&\leq\frac12 \Lambda |B_r\setminus B_{r-\sqrt{r}}|\,r^{-1}
	-\int_{\O \cap B_{r-\sqrt{r}}}F(x,1)
	-\left(\min_{\ol\O\times[0,1]}F\right)\vert   B_r \setminus B_{r-\sqrt{r}} \vert ,
	\end{split}\]
	where $\Lambda$ is given by \eqref{A}. We eventually infer the existence of a constant $C$ independent of $r$ such that
	\Fi{E<}
	\forall  r>1,\quad
	\mc{E}_r(\ul u_r)\leq C r^{N-1/2}-
	\int_{\O \cap B_{r-\sqrt{r}}}F(x,1).
	\Ff
	
	\medskip
	{\em Step 2. Lower bound for the average of $F(x,\ul u_r)$.}\\
First, observe that we have
	\[
	\mc{E}_r( \ul u_r) \geq
	-\int_{\O\cap B_{r-\sqrt{r}}}F(x, \ul u_r)-\vert B_r\setminus B_{r-\sqrt{r}}\vert \max_{\ol\O \times [0,1]} f.
	\]
	Combining this with \eqref{E<}, we get
	
	\begin{equation}\label{mean}
	\int_{\O \cap B_{r-\sqrt{r}}} (F(x,1)- F(x, \ul u_r) ) -\vert B_r\setminus B_{r-\sqrt{r}}\vert \max_{\ol\O \times [0,1]} f \leq  C r^{N-1/2}.
	\end{equation}
%
%
	The inequality \eqref{mean} holds for all $r>1$, hence, using the fact that
	$|\O\cap B_r|\sim\frac{\vert\mc{C}\vert}{\prod L_{i}}|B_r|=\frac{\vert\mc{C}\vert}{\prod L_{i}}|B_1|r^N$ as $r$ goes to $+\infty$, dividing \eqref{mean} by $\vert \O \cap B_{r-\sqrt{r}} \vert$ and taking the limit $r \to +\infty$, 
	we eventually infer that
	\begin{equation}\label{Mk}
	\liminf_{r\to+\infty}\fint_{\O\cap B_{r-\sqrt{r}}} (F(x, \ul u_r)-F(x,1))
	\geq 0.
	\end{equation}

	\medskip
	{\em Step 3. Convergence of the maxima to $1$}.\\
We show now that  \eqref{Mk} implies that 
\Fi{urto1}
\sup_{\O\cap B_{r-\sqrt{r}}}\ul u_r   \underset{r \to +\infty}{\longrightarrow} 1.
\Ff
We proceed by contradiction : assume that there is a diverging sequence 
$(r_{n})_{n\in\N}$ such that (we recall that $\underline{u}_r$ satisfies $\max \underline{u}_r< 1$)
$$M := \sup_{\su{n\in\N}{x\in\O\cap B_{r_n-\sqrt{r_n}}}}\ul u_{r_n}(x)<1.$$ 
Then,  \eqref{Mk}  implies that
$$
0 \leq  \liminf_{n\to+\infty}
 \fint_{\O\cap B_{r_n-\sqrt{r_n}}}
(F(x,\ul u_{r_n}) -F(x,1)) \leq \max_{\ol\O \times [0,M]}(F(x,s) - F(x,1)).
$$
This contradicts the hypothesis \eqref{strict}.

\medskip
	{\em Step 4. $\ul u_{r}$ is large on a large set.}\\	
	Consider a sequence of radii $(r_{n})_{n\in\N}$ 
	diverging to $+\infty$ and satisfying \eqref{nonparallel}.
	For $n\in\N$, let $x_{n}\in \overline{\O\cap B_{r_{n}-\sqrt{r_{n}}}}$ be such that 
	$\ul u_{r_{n}}(x_n) = \max_{\overline{\O\cap B_{r_n-\sqrt{r_n}}}}\ul u_{r_n}$.
	Then let $z_n\in\prod_{i=1}^{N} L_{i}\Z$ be
	such that $x_{r_n}-z_n\in\mc{C} $. Finally, define $u^n(x):= \ul u_{r_n}(x+z_n)$. For any compact set $K \subset \ol\O$, these functions are well defined in $K$, for $n$ large enough, because $r_{n} - \vert z_{n} \vert > \sqrt{r_{n}} - \sqrt{\sum_{i=1}^{N}L_{i}^{2}}$. Hence, owing to the partial boundary estimates (see, e.g., \cite[Theorem 6.30]{GT})
	they converge (up to subsequences) locally uniformly in 
	$\ol\O$ to a solution $u^*\leq1$ of \eqref{stationary}.
	Furthermore, by the choice of $x_{n}$ and \eqref{urto1}, 
%
 we have that $\max_{\ol{\mc{C}}}u^*=1$. Proceeding exactly as in the proof of Lemma \ref{lem:ur},
	by means of Hopf's lemma and strong maximum principle, we eventually infer that
	$u^*\equiv1$. This shows that $\ul u_{r_n}(\cdot+z_n)$ converges to $1$ locally uniformly in $\ol \O$ as $n$ goes to $+\infty$.
	
\medskip	
{\em Step 5. Conclusion.}\\
We are now in position to conclude the proof. Take a compact set $K\subset \overline{\O}$ and $\varepsilon>0$ and $\eta > \theta$. First, owing to the fourth step, we can find $\tilde{r}>0$ and $z\in\prod_{i=1}^{N} L_{i}\Z$ such that
$$
\forall x \in K, \quad \ul u_{\tilde{r}}(x + z) > 1-\varepsilon.
$$

Now, arguing as in the proof of Theorem \ref{th persistence} above, we can find $r>0$ such that, if $u(t,x)$ is the solution if \eqref{evol} arising from an initial datum $u_{0} \geq \eta$ on $\ol\O \cap B_{r}$, then there is $T>0$ such that 
$$
u(T,\cdot) \geq \ul u_{\tilde{r}}(\cdot +z).
$$
Because $\ul u_{\tilde{r}}(\cdot +z)$ extended by $0$ out of $B_{r}(-z)$ is a generalized stationary subsolution of \eqref{evol}, we have that $u(T+t,\cdot) \geq \ul u_{\tilde{r}}(\cdot +z)$, for every $t\geq 0$. The result follows.
\end{proof}

\section{Invasion}\label{section invasion}

This section is dedicated to the proof of the invasion results Theorems \ref{invasion general}, \ref{inv ptf} and their Corollaries  \ref{invasion combustion}, \ref{invasion bistable} and \ref{invasion bistable 2}.

\subsection{Proofs of Theorems \ref{invasion general} and \ref{inv ptf}}\label{sec:invasion}
The idea to prove Theorem \ref{invasion general} is mainly geometrical: roughly speaking it reduces to approaching  
front-like initial data by a sequence of compactly supported data.
Let us preliminarily observe the following fact, that will be used several times in the sequel.
\begin{lemma}\label{ODE}
	The unique bounded solution $u$ to \eqref{evol} satisfying $\inf u>\theta$, with $\theta$ defined in~\eqref{deftheta},
	is 	$u\equiv1$.
\end{lemma}

\begin{proof}
	Let $\ul v$, $\ol v$ be the solutions to the following ODEs:
	$$\dot{\ul v} = \min_{x\in \ol\Omega}f(x,\ul v),\qquad
	\dot{\ol v} = \max_{x\in \ol\Omega}f(x,\ol v),$$ 
	with initial data $\ul v(0) =\inf u>\theta$ and $\ol v(0) =\max\{1,\sup u\}$.
	These functions are respectively a sub and a supersolution to \eqref{evol}, and the same is true for their translations
	$\ul v(\.+T)$, $\ol v(\.+T)$ defined in $[-T,+\infty)$, for any $T\in\R$.
	We then deduce from the parabolic comparison principle that
	$$\ul v(t+T)\leq u(t,x)\leq\ol v(t+T)\quad\text{ for all } \ T\in\R, \ t\geq-T \ \text{ and } \ x\in\O.$$
	Now, because $f(\.,s)$ is positive for $s\in(\theta,1)$ by the definition~\eqref{deftheta} of $\theta$,
	and because $f(\cdot,s)$ is negative for $s>1$, it is clear that $\ul v(t)\to1$ and $\ol v(t)\to1$ as $t$ goes to $+\infty$.
	As a consequence, letting $T$ go to $+\infty$ in the above inequalities yields $u\equiv1$.	
\end{proof}	

\begin{proof}[Proof of Theorem \ref{invasion general}]

It is straightforward to see that property $(ii)$ implies $(i)$: every front-like datum 
(in the sense of \eqref{def f l}) is, up to a suitable translation, 
larger than any $\eta<1$ in any bounded subset of $\O$.
The proof of the reverse implication is split into four steps.

\medskip
\emph{Step 1. Reducing to an equivalent property.}\\
Take $\eta,\eta'$ satisfying $\theta<\eta'<\eta<1$. For $n \in \mathbb{N}$, $n \geq 2$, we let $u_{0}^{n}$ denote a 
non-negative function compactly supported in $\ol\O \cap B_{n}$ such that $u_{0}^{n} = \eta$ on $\ol{\O \cap B_{n-1}}$. We call $u^{n}$ the solution of \eqref{evol} arising from the initial datum $u_{0}^{n}$. We claim that there exists $n>0$ such that
\begin{equation}\label{2 rest}
	\text{for every compact } K \subset \ol\O,\quad \liminf_{t\to+\infty}\left(\min_{x\in
	K}u^{n}(t,x)\right)\geq\eta'.
\end{equation}
Before proving this claim, let us show how it entails property $(ii)$. Consider a diverging sequence $(t_k)_{k\in\N}$.
By usual parabolic estimates, the functions $u^{n}(\.+t_{k},\.)$
converge as $k$ goes to $+\infty$ (up to subsequences)
locally uniformly in $\R\times\ol\O$ 
to an entire solution $u^{\infty}$ of \eqref{evol}. Using \eqref{2 rest} we find that
$u^{\infty}(t,x) \geq \eta^{\prime}>\theta$ for all $t \in \mathbb{R}$ and $x\in \Omega$. 
It then follows from Lemma \ref{ODE} that $u^\infty\equiv1$.
This shows that $u^n(t,x)\to1$ as $t$ goes to $+\infty$, locally uniformly in $x\in\ol\O$, that is the invasion property. 

We have derived the invasion property for the initial datum $u_{0}^{n}$, 
provided \eqref{2 rest} holds, 
and then by comparison for all initial data larger than
$u_{0}^{n}$. This is precisely property $(ii)$.
It remains to prove that \eqref{2 rest} holds for $n$ sufficiently large. 
We argue again by contradiction, assuming that for any $n\geq 2$, 
\begin{equation}\label{absurd}
	\exists K_{n} \text{ compact subset of } \ol\O,\quad\liminf_{t\to+\infty}
	\left(\min_{x\in K_{n}}u^{n}(t,x)\right)<\eta'.
\end{equation}

\medskip
\emph{Step 2. Lower bound on the expansion of the level sets.}\\
For $n\geq 2$ we define
$$T_n:=\inf\left\{t\geq0\,:\,\exists x\in \ol{\O\cap B_{\sqrt{t}}},\ u^{n}(t,x)\leq \eta'\right\}.$$
	Observe that \eqref{absurd} implies that the above set is nonempty because 
	$K_n\subset \ol{\O\cap B_{\sqrt{t}}}$ for $t$ large.
	We have that
	$$\forall t\in(0,T_n),\ \forall x\in \ol{\O\cap B_{\sqrt{ t}}},\quad
	 u^{n}(t,x)>\eta',$$
	and there exists $x_n\in\ol{\O\cap B_{\sqrt{T_n}}}$ such that $u^{n}(T_n,x_n)=\eta'$.

Let us show that $T_n\to+\infty$ as $n$ goes to $+\infty$. Because the initial datum $u^{n}_{0}$ satisfies the boundary condition of \eqref{evol} on $\ol\O \cap B_{n-1}$, we can apply the parabolic estimates (see \cite[Theorems 5.2, 5.3]{La}) to get that  $u^{n}$ converges locally uniformly in $[0,+\infty)\times\ol\O$ to a solution $v$ of \eqref{evol} satisfying $v(0,x)\geq\eta$ for $x\in\ol\O$. In particular, $v(t,x)\geq\eta$ for all $t\geq0$ because 
$\eta \in(\theta,1)$
is a subsolution to \eqref{evol} by \eqref{deftheta}. This local uniform convergence implies that, for every $T>0$, we can find $n$ large enough such that
$$
 \forall t \in [0,T], \ \forall x \in \O\cap\overline{B_{\sqrt{T}} }, \quad u^{n}(t,x) >\eta^{\prime}.
$$
Hence, $T_n \geq T$, for $n$ large enough. This means that $T_n\to +\infty$ as $n$ goes to $+\infty$.


\medskip
\emph{Step 3. Reducing to a front-like entire solution.}\\
Consider the sequence $(z_n)_{n\in\N}$ in $\prod_{i=1}^{N}L_{i}\Z$ for which $y_{n} := x_n-z_n\in\mathcal{C}$. Then define
		$$w_n(t,x):=u^n(t+T_n,x+z_n).$$ 
		By the periodicity of the problem, the functions $w_n$ are solutions
		to \eqref{evol} for $t>-T_n$ and satisfy 
		$$w_n(0,y_n)=\eta',$$
		\Fi{wn>}
		\forall t\in[-T_n,0],\ \forall x\in \O\cap B_{\sqrt{t+T_n}}(-z_n),
		\quad w_n(t,x)\geq \eta'.
		\Ff
Because $T_n\to +\infty$ as $n$ goes to $+\infty$ by the previous step,
the sequence $(w_n)_{n\in\N}$ converges (up to subsequences) locally uniformly
		to an entire solution $w_\infty$ of \eqref{evol}, i.e., a solution for all times $t\in\R$. Observe that $w_\infty$ satisfies $w_\infty(0,y)=\eta'$, where $y$ is the limit of (a subsequence of) $(y_n)_{n\in\N}$. Furthermore, defining for $t\in \mathbb{R}$ 
$$
H_t :=\O\cap \bigcup_{M\geq 1} \bigcap_{n\geq M} B_{\sqrt{t+T_n}}(-z_n)
$$
(by convention, we set $B_{\sqrt{\tau}} = \emptyset$ if $\tau\leq0$)
we see that
$$
 \forall t\leq0, \ \forall x \in H_t, \quad w_\infty(t,x)\geq \eta'.
$$
Assume first that (a subsequence of) $( x _n )_{n\in \mathbb{N}}$ is bounded. 
Then, so is $( z _n )_{n\in \mathbb{N}}$ and thus
$B_{\sqrt{t+T_n}}(-z_n)$ invades $\R^N$ as $n$ goes to $+\infty$, which implies that
$H_t = \Omega$ for all $t$. As a consequence, $w_\infty$ satisfies 
$$
\forall t\leq0, \ \forall x \in \Omega, \quad w_\infty(t,x)\geq \eta',
$$
and therefore $w_\infty\geq\eta'$ for all $t\in\R$ and $x \in \Omega$ by comparison
with the subsolution identically equal to $\eta'$.
Lemma~\ref{ODE} eventually yields $w_\infty\equiv1$, which is impossible because $w_\infty(0,y)=\eta'$. 
Consider now the other possibility:
$$
\vert x_{n}\vert \underset{n \to +\infty}{\longrightarrow} +\infty.
$$
Let us show that there is $e\in \mathbb{S}^{N-1}$ such that
$$
 \forall t \leq 0, \quad P := \left\{ x\in\O\ :\  x\cdot e  < y\cdot e \right\} \subset H_t.
$$
Let $e\in \mathbb{S}^{N-1}$ be such that, up to extraction, $\hat{x}_n := \frac{x_n}{\vert x_n \vert} \to e$ as $n$ goes to $+\infty$. Take $t\leq 0$ and $x \in P$. Then, there is $\e>0$ such that $x\cdot e < y\cdot e -\e$. For $n \in \mathbb{N}$, we have
\begin{equation*}
\begin{array}{lll}
\vert x +z_{n}\vert^{2} &=& \vert x -y_{n} +x_{n}\vert^{2} \\
&=& \vert x -y_n\vert^{2} + \vert x_n \vert^{2} +2(x-y)\cdot x_n +2(y-y_n)\cdot x_n \\
&<& \vert x -y_n \vert^{2} + \vert x_n \vert^{2} -2\e \vert x_n \vert+ 2(x-y)\cdot (x_n-\vert x_n \vert e)  +2(y-y_n)\cdot x_n \\
&\leq&  \left( \frac{\vert x -y_n \vert^{2}}{\vert x_n \vert}  -2\e  +2(x-y)\cdot (\hat{x}_n-  e) +2(y-y_n)\cdot \hat{x}_n \right) \vert x_n \vert +\vert x_n \vert^{2}.
\end{array}
\end{equation*}
Because $\vert x_n \vert \to +\infty$  and $\hat{x}_{n} \to e$ as $n$ goes to $+\infty$,  we have that 
$$
\limn\left( \frac{\vert x -y_n \vert^{2}}{\vert x_n \vert}  -2\e 
+2(x-y)\cdot (\hat{x}_n-  e)  +2(y-y_n)\cdot \hat{x}_n\right)
=-2\e.$$ 
Therefore, recalling that
$\vert x_n \vert \leq \sqrt{T_n}$, we derive for $n$ large enough
\begin{equation*}
\vert x +z_{n}\vert^{2} \leq T_n  -\e \sqrt{T_n} <  T_n + t .
\end{equation*}
This means that $x \in H_{t}$. Hence
$$
\forall t\leq0, \ \forall x \in P, \quad w_\infty(t,x)\geq \eta' \ \text{ and } \ w_\infty(0,y)=\eta' .
$$

Now, because $\eta'\in(\theta,1)$,
\cite[Lemma 1]{D} ensures that the functions $w_{\infty}(t,\cdot)$ are actually 
front-like as $x\cdot e\to-\infty$ \emph{uniformly} with respect to $t<0$, in the sense that
\Fi{wfl}
\lim_{x\cdot e \to -\infty} \left(\inf_{t<0} w_{\infty}(t,x)\right) =1.
\Ff

\medskip
\emph{Step 4. Proof of the lower bound \eqref{2 rest}.}\\
Here we use property $(i)$. 
Owing to the previous step, the function $\underline{w}_{0}$ defined by 
$$\underline{w}_{0}(x) := \chi(x\.e)\big(
\inf_{t<0}w_{\infty}(t,x)\big),$$ with $\chi$ smooth, decreasing and satisfying
$\chi(-\infty)=1$, $\chi(+\infty)=0$, fulfils the front-like condition 
\eqref{def f l}. 
Therefore, by property $(i)$, invasion occurs for the solution 
$\underline{w}$ of~\eqref{evol} with initial datum $\underline{w}_{0}$. Now, for $m>0$, there holds
$$
 \forall x \in \Omega, \quad \underline{w}_{0}(x) \leq w_{\infty}(-m,x),
$$
and then, by comparison, $
\underline{w}(m,y) \leq w_{\infty}(0,y) = \eta^{\prime}$.
This contradicts the invasion property of $\underline{w}$.
We have proved that \eqref{2 rest} holds for $n$ large enough. 
\end{proof}

We now prove our second main result concerning the invasion property.

\begin{proof}[Proof of Theorem \ref{inv ptf}]
We prove the result by showing that property $(ii)$ of
Theorem~\ref{invasion general}
can still be derived by if one replaces property $(i)$
with Hypothesis \ref{H2}.
Recall that the only step of the previous proof which makes use of property $(i)$ is step 4.
Let us check that it holds true under Hypothesis \ref{H2}.
Let $v$ be a pulsating traveling front in the direction $e$ with a positive speed $c$.
Using the fact that $c>0$ and that $w_{\infty}$ satisfies \eqref{wfl}, together with
the nonincreasing monotonicity of $s \mapsto f(x,s)$ in some neighbourhoods
of $0$ and~$1$, one can prove that 
$$
\forall t\in \mathbb{R}, \ \forall x\in \O, \quad v(t,x) \leq w_{\infty}(t,x) .
$$
This is classical in the homogeneous case, in our heterogeneous framework we can invoke \cite[Lemmas 1, 2]{D}.
Of course the function $v(t,x)$ can be replaced by any temporal-translation 
$v(t+\tau,x)$, and thus, because $v$ fulfils \eqref{ptf} with $c\neq0$,
we can assume without loss of generality that $v(0,y)>\eta'$. This is impossible because
$w_{\infty}(0,y)=\eta'$. This proves \eqref{2 rest}, concluding the step 4.
\end{proof}

\begin{remark}\label{vitesse positive}
The validity of the invasion property for a compactly supported non-negative initial datum $u_0$ 
satisfying $\sup u_0<1$ readily implies that the invasion takes actually
 place with at least a linear speed. Indeed the associated solution $u$ satisfies, for some $T>0$, 
 $$
 \forall x \in  \O,\ \forall k\in \prod_{i=1}^N\{-L_i,0,L_i\},
 \quad u(T,x)\geq 
 u_0(x+k),
 $$
 and thus, since the spatial translations of $u$ by $L_1\Z\times\cdots\times L_N\Z$
 are still solutions of~\eqref{evol}, using the comparison principle we get,
 by iteration, 
  $$\forall n\in\N,\
 \forall x \in  \O,\ \forall k\in \prod_{i=1}^N\{-nL_i,\dots,0,\dots,nL_i\},
 \quad u(nT,x)\geq 
 u_0(x+k).
 $$
 Now, owing to the invasion property,
 for any $\e>0$ we can find $\tau>0$ such that $u(t,x)>1-\e$ for $t\geq\tau$ and
 $x$ in the periodicity cell $\mc{C}$. We eventually infer that
 $u(t,x)>1-\e$ for $t\geq nT+\tau$ and
 $x\in\ol\O\cap([-nL_1,nL_1]\times\cdots\times[-nL_N,nL_N])$.
 This means that the upper level set $\{u>1-\e\}$ propagates with at least a linear speed.
%
\end{remark}

\subsection{Applications of the invasion result}\label{sec:cor}

In this section, we make use of Theorem \ref{inv ptf}  to derive Corollary \ref{invasion combustion} and two other results that apply to the bistable case.
This essentially reduces to check that Hypothesis~\ref{H2} holds.

Corollary \ref{invasion combustion} applies to non-negative nonlinearities, i.e., satisfying \eqref{combustion}, such as the Fisher-KPP nonlinearity $f(u) = u(1-u)$ or the Arrhenius nonlinearity $f(u) = e^{-\frac{1}{u}}u(1-u)$.
To derive it, we shall need the following result from Berestycki and Hamel.
\begin{proposition}[\text{\cite[Theorem 1.13]{pulsating}}]\label{prop pulsating}
Assume that $f$ satisfies \eqref{combustion} and
\begin{equation}\label{type1}
\left\{
\begin{array}{llc}
\exists \vartheta \in (0,1) ,\  f(x,s) =0 \text{ if } s\leq \vartheta, \\
\forall s>\vartheta , \ \exists x \in \ol\O \text{ such that } f(x,s)>0,
\end{array}
\right.
\end{equation}
and that $q$ satisfies \eqref{type1 q}. 
Then, Hypothesis \ref{H2} holds for the problem \eqref{evol}.
\end{proposition}

\begin{proof}[Proof of Corollary \ref{invasion combustion}]
Assume that $q$ and $f$ satisfy the hypotheses of Corollary \ref{invasion combustion}.
 We cannot apply Theorem \ref{inv ptf} directly to $f$, because it may not be 
 non-increasing in a neighbourhood of $s=0$.
 To overcome this, we take $\e \in (0,1-\theta)$, with $\theta$ given by \eqref{deftheta}, 
 and we define a nonlinearity $\tilde{f}$ independent of $x$
 satisfying

\begin{equation*}
\left\{
\begin{array}{ll}
\tilde{f}(s)=0 \quad  &\text{ for }\ s\in [0,\theta + \e ],  \\
0<\tilde{f}(s) \leq \min_{x\in \ol\O}f(x,s)\quad  &\text{ for }\ s \in (\theta+\e,1).
\end{array}
\right.
\end{equation*}
Applying Proposition \ref{prop pulsating} to $\t f$ (with $\vartheta = \theta + \e$), we 
deduce that Hypothesis \ref{H2} is verified for the problem \eqref{evol} with the nonlinearity $\tilde{f}$. Now, we can apply Theorem \ref{inv ptf} to get that, for $\eta \in(\theta+\e , 1)$, there is $r>0$ such that, if 
$$
 u_0 \geq \eta \quad \text{ in } \ x \in B_r,
$$
then the solution $\tilde{u}(t,x)$ of \eqref{evol} with nonlinearity $\tilde{f}$ and initial datum $u_{0}$ satisfies the invasion property. By comparison, the same holds true
for the solution of \eqref{evol} with initial datum $u_0$ but with nonlinearity $f$, 
because $\tilde{f}\leq f$. 
As $\e$ can be chosen arbitrarily close to $0$, this yields the result.
\end{proof}

In the bistable case, some sufficient conditions for the existence of pulsating fronts
are provided in the whole space with $1$-periodic terms 
(i.e., satisfying our usual definition of periodicity with $L_{1}, \ldots, L_{N} =1$)
by Ducrot \cite[Corollary 1.12]{Ducrot}  and Xin 
\cite[Theorem 2.2]{Xin3}. These results, combined with Theorem \ref{inv ptf}, 
directly yield the~following.
\begin{corollary}\label{invasion bistable}
Consider the equation
\begin{equation}\label{bistable Ducrot}
\partial_{t}u = d\Delta u +r(x)u(u-a(x))(1-u),\quad t >0, \ x\in \mathbb{R}^{2},
\end{equation}
where $a, r \in C^{\alpha}(\mathbb{R}^{2})$, for some $\alpha \in (0,1)$, are 
$1$-periodic and satisfy
$$
 \forall x\in \mathbb{R}^{2},\quad 0<a(x)< 1 \ \text{ and }\ r(x)>0,
$$
and
$$
\overline{\theta} := \frac{\int_{[0,1]^{2}}r(x)a(x)dx}{\int_{[0,1]^{2}}r(x)dx}<\frac{1}{2}.
$$
Then, there is $d_{0}\geq 1$ large enough such that 
properties $(i)$-$(ii)$ of 
Theorem~\ref{invasion general} hold provided $d\geq d_{0}$.
\end{corollary}

\begin{corollary}\label{invasion bistable 2}
Consider the equation
\Fi{Xin}
\partial_{t}u = \nabla ( A(x) \nabla u) + q(x)\cdot \nabla u + u(1-u)(u -\theta), \quad t>0, \ x\in \mathbb{R}^{N},
\Ff
where $A$ is a uniformly elliptic smooth matrix field and $q$ is a smooth
vector field which are
$1$-periodic and $\theta \in (0,\frac{1}{2})$. Then, there is $\delta>0$ and 
$s >N+1$, such that properties $(i)$-$(ii)$ of 
Theorem~\ref{invasion general} hold provided 
\Fi{Abar}
\left\|  A - \int_{[0,1]^{N}}A \right\|_{H^{s}([0,1]^{N})} < \delta \ \text{ and } \ \left\| q \right\|_{H^{s}([0,1]^{N})} < \delta.
\Ff
\end{corollary}


%


\section{Estimates on the spreading speed}\label{estimate}

This section is devoted to the proof of Theorem \ref{invasion2}. As already 
mentioned, this result applies to equations set in the whole space $\mathbb{R}^{N}$ with a nonlinearity that satisfies \eqref{combustion}. 
It provides a lower estimate on the speed of invasion.

 The philosophy of this section differs from the previous one in that we shall build 
 ``explicitly" a subsolution that invades the space with some given speed.
We start with a technical lemma.

\begin{lemma}\label{propdim}
Assume that $f$ is independent of $x$ and satisfies \eqref{combustion}. 
Then, for any $0<\lambda \leq \Lambda$, there exist $H \in (\theta, 1)$, $L>0$ and a non-increasing function $h \in W^{2,\infty}(\R)$ such that
\begin{equation*}
\left\{
\begin{array}{lll}
h(z) = H  \quad &\text{ in }\ (-\infty,0] , \\
h(z) = 0  \quad &\text{ in }\ [L,+\infty), \\
\end{array}
\right.
\end{equation*}
and
\begin{equation}\label{prop h}
Ah''+Bh'+f(h)\geq 0 \quad \text{ in }\ [0,L], \ \text{ for } \ A \in [\lambda, \Lambda], \ B \leq  \frac{\lambda}{\sqrt{\Lambda}} \sqrt{R(f)}. 
\end{equation}

\end{lemma}

%
%
%
%


Before proving Lemma \ref{propdim}, we show how it yields Theorem \ref{invasion2}.

\begin{proof}[Proof of Theorem \ref{invasion2}.] 
First of all, up to replacing $f(x,s)$ with $\min_{x\in \ol\O}f(x,s)$, it is not restrictive to assume 
that $f$ is independent of~$x$ (and still satisfies \eqref{combustion}).
We have that $f(s)>0$ for $s\in(\theta,1)$.

By hypothesis, there holds
$$
w^{\star} := \frac{\lambda}{\sqrt{\Lambda}} \sqrt{R(f)}-\limsup_{\vert x \vert \to +\infty}q(x)\cdot\frac{x}{\vert x\vert} >0.
$$
Fix $\bar c \in (0, w^{\star})$. Then take $\rho>0$ in such a way that 
$$
 \forall  x  \in B_{\rho}^{c}, \quad \frac{N\Lambda}{\rho}+\bar c+q(x)\cdot\frac{x}{\vert x\vert}< \frac{\lambda}{\sqrt{\Lambda}} \sqrt{R(f)}.
$$ 
Now, let $H,L$ and $h$ be the constants and the function 
provided by Lemma \ref{propdim}, associated with $\lambda,\Lambda$ 
given by \eqref{A}.
Consider the function
$$
v(t,x) := h(\vert x \vert -\bar c t - \rho).
$$
We claim that $v$ is a subsolution of \eqref{evol2}. Equation 
 \eqref{evol2} trivially holds outside the region $0<\vert x \vert -\bar c t - \rho<L$, where
 $v$ is constant, hence it suffices to show~that
$$\forall 
t>0,  \     \forall x\in B_{\rho +L+\bar  ct}\setminus B_{\rho + \bar ct},\quad
\partial_{t}v-Tr(AD^{2}v)-q\cdot\nabla v-f(v)\leq 0.
$$
Direct computation shows that this is equivalent to have, for such $t$ and $x$,
$$
A h^{\prime\prime}(\vert x \vert -\bar c t - \rho)+
 B h^{\prime}(\vert x \vert -\bar c t - \rho)+f(h(\vert x \vert -\bar c t - \rho))\geq 0,
$$
where 
$$A=\left(A(x)\frac{x}{\vert x\vert}\cdot \frac{x}{\vert x\vert}\right),\qquad
B=\left(\frac{Tr(A(x))}{\vert x\vert}-
\frac{1}{\vert x\vert}\left(A(x)\frac{x}{\vert x\vert}\cdot \frac{x}{\vert x\vert} \right)+
q(x)\cdot\frac{x}{\vert x\vert}+\bar c\right).$$
Observing that $\lambda\leq A\leq \Lambda$ 
and that
$$
B
 \leq \frac{N\Lambda}{\rho}+q(x)\cdot \frac{x}{\vert x \vert} +k
 < \frac{\lambda}{\sqrt{\Lambda}} \sqrt{R(f)}
$$
because $\vert x \vert \geq \rho$,
the above inequality holds thanks to \eqref{prop h}. 
This shows that $v$ is a subsolution of \eqref{evol2}. 

Now, arguing as in the proof of Theorem \ref{th persistence}, for any $\eta>\theta$ we can
find $r,T>0$ such that every solution $u$ of \eqref{evol2}
with a non-negative initial datum $u_0\geq \eta$ in $B_r$ satisfies
$u(T,x)>v(0,x)$ for $x\in\R^N$. 
Hence, by comparison, $u(t+T,x)>v(t,x)$ for $t>0$, $x\in\R^N$. 
Consider such a solution $u$ and take $c\in(0,\bar c)$. 
We claim that
$$
\inf_{\vert x \vert \leq ct} u(t,x) 
\underset{t \to +\infty}{\longrightarrow}1.
$$
 We proceed by contradiction:  assume that there is $\e \in (0,1)$, 
 a diverging sequence $(t_{n})_{n\in\N}$ and a sequence
 $(x_{n})_{n\in\N}$ in $\mathbb{R}^{N}$  such that $\vert x_{n}\vert \leq ct_{n}$ and $u(t_{n},x_{n}) \leq 1-\e$. We define the sequence of translations $u_{n}(t,x) := u(t+t_{n},x+z_{n})$, with $(z_n)_{n\in\N}$ in $\prod_{i = 1}^{N}L_{i}\Z$ satisfying 
 $x_n-z_n\in \mc{C}=\prod_{i = 1}^{N}[0,L_{i})$. 
 This sequence converges (up to extraction) to an entire solution 
 $u_{\infty}$ of \eqref{evol2}. 
 It satisfies $\min_{x\in\ol{\mc{C}}}u_{\infty}(0,x) \leq 1-\e$ and, for $t\in\R$ and $x\in\R^N$,
\[\begin{split}u_{\infty}(t+T,x) &=\lim_{n\to+\infty}u(t+T+t_{n},x+z_{n})
  \geq \lim_{n\to+\infty}v(t+t_{n},x+z_{n})\\
  &=  \lim_{n\to+\infty}h(\vert x+z_{n} \vert -\bar c (t+t_n) - \rho).
  \end{split}\]
 The latter term is equal to $H$ because $\vert z_{n}\vert \leq 
 \vert x_{n}\vert + \sqrt{L_1\cdots L_N} \leq ct_{n}+ \sqrt{L_1\cdots L_N} $ and 
 $c<\bar c$. Hence $u_{\infty}$ is everywhere larger than $H>\theta$.	
Therefore, again by the same argument as in the proof of Theorem \ref{th persistence},
we get $u_\infty\geq1$, which contradicts $\min_{x\in\ol{\mc{C}}}u_{\infty}(0,x) \leq 1-\e$. This yields the result.
\end{proof}

We now turn to the proof of Lemma \ref{propdim}.

\begin{proof}[Proof of Lemma \ref{propdim}]
By definition of $R(f)$ and $\theta$, there exist
$\theta<K<H<1$ such that $R(f)=(H-K)\min_{[K,H]}f$. 
For $0<z_{1}<z_{2}$ and $ \beta, \gamma, \mu>0$
that will be chosen later, we define 
the function $h$ on each interval $(-\infty,0]$, $[0,z_{1}]$, $[z_{1},z_{2}]$
as follows:
\begin{equation}\label{fcth}
\left\{
\begin{array}{llc}
h(z)=H \quad & \text{ for } \ z\in (-\infty,0], \\
h(z)=H-\frac{\gamma}{2}z^{2} \quad & \text{ for } \ z\in [0,z_{1}],\\
h(z)=\mu(z_{2}-z)^{\beta} \quad &\text{ for } \ z \in [z_{1},z_{2}].
\end{array}
\right.
\end{equation}
We want to find
$z_{1},z_{2}, \beta, \gamma, \mu $ so that $h\in W^{2,\infty}(\R)$ and \eqref{prop h} holds.
We further impose $h(z_{1})=K$. 
We shall take $\beta \geq 2$. Then, to have $h\in W^{2,\infty}$ we need
$$K=H-\frac{\gamma}{2}z_{1}^{2},\qquad K=\mu(z_{2}-z_{1})^{\beta},$$ 
$$\gamma z_{1}=\mu\beta(z_{2}-z_{1})^{\beta-1}=K\beta(z_{2}-z_{1})^{-1}.$$

Now, let us see what conditions are needed to have \eqref{prop h}. 
Call for short $\ol B :=  \frac{\lambda}{\sqrt{\Lambda}} \sqrt{R(f)}$. On one hand, 
if $z \in [0,z_{1}]$, property \eqref{prop h} holds as soon as
$$
-\gamma \Lambda-\gamma\ol B z_1+f(h(z))\geq0.
$$
Then, using the fact that $h(z)\in[K,H]$ for $z\in [0,z_{1}]$, it is sufficient to have$$
-\gamma\big(\Lambda+ \ol B z_{1}\big)+\min_{[K,H]}f\geq0.
$$
On the other hand, for $z\in [z_{1},z_{2}]$, property \eqref{prop h} holds as soon as (recall that $f\geq0$)
$$
\lambda(\beta-1)-\ol B(z_{2}-z_{1}) \geq 0.
$$
For notational simplicity, we shall write $\Delta=z_{2}-z_{1}$. Summing up, $h$ satisfies 
\eqref{prop h} on $\R\setminus\{z_1,z_2\}$ provided
we can find $ z_{1},\Delta,\beta,\mu, \gamma$ such that 

\begin{equation}\label{alg}
\left\{
\begin{array}{llc}
\lambda(\beta-1)-\ol B\Delta \geq0, \\
\gamma\big(\Lambda+ \ol B z_{1}\big)\leq\min_{[K,H]}f, \\
\frac{\gamma}{2}z_{1}^{2}=H-K, \\
\mu\Delta^{\beta}=K, \\
\gamma z_{1}=K\beta\Delta^{-1}.
\end{array}
\right.
\end{equation}
Let us show that this is solvable. We leave $\beta$ as a free parameter and take
$$
\gamma = \left( \frac{\ol B \beta K}{\lambda\sqrt{2(H-K)}(\beta-1)} \right)^{2},
$$
 and then
 $$
 z_{1} = \frac{\sqrt{2(H-K)}}{\sqrt{\gamma}},\qquad 
 \Delta = \frac{\beta K}{\sqrt{2\gamma (H-K)}},\qquad \mu = \frac{(2\gamma (H-K))^{\frac{\beta}{2}}}{\beta^{\beta}K^{\beta-1}}.
 $$
Direct computation shows that all the equations of \eqref{alg} are satisfied, 
with the possible exception of the second one. 
Let us show that the second equation holds as well, provided $\beta$ is sufficiently
large. To do so, we observe that
$$
\gamma \underset{\beta \to +\infty}{\longrightarrow} \left( \frac{\ol B  K}{\lambda\sqrt{2(H-K)}} \right)^{2} \ \text{ and } \ z_{1} \underset{\beta \to +\infty}{\longrightarrow} \frac{ 2\lambda(H-K)}{ \ol B  K},
$$
whence, recalling the expression of $\ol B$, 
$$
\gamma\big(\Lambda+ \ol B z_{1}\big)    \underset{\beta \to +\infty}{\longrightarrow} \frac{  K^2}{2\Lambda }    \left(\Lambda+   \frac{2\lambda(H-K)}{K}\right)
 \min_{[K,H]}f .
 $$
Now, because $K,H \in [0,1]$ and $\lambda \leq \Lambda$, we find that
$$
\frac{  K^2}{2 \Lambda}    \left(\Lambda+   \frac{2\lambda(H-K)}{K}\right) \leq  \frac{1}{2}  K\left(2H- K \right) <1.
$$
It follows that, choosing $\beta$ large, 
the second equation of \eqref{alg} is verified too.
\end{proof}



Now that we have given some sufficient conditions that ensure that invasion occurs, we focus to the case where $f$ is a bistable nonlinearity. In this case, Hypothesis~\ref{H2} is not known to hold in general, thus we cannot apply Theorem \ref{inv ptf}. In Section \ref{section bistable}, we show how the geometry of the domain can either totally block or allow the invasion. In Section \ref{Oriented invasion}, we present a new phenomenon that we call \emph{oriented invasion}.

\section{Invasion and blocking in domains with periodic holes}\label{section bistable}

\subsection{Invasion}\label{invasion}

This section is dedicated to the proof of Theorem \ref{th invasion bistable}. In the whole section, $f$ will denote an unbalanced bistable nonlinearity, i.e, satisfying \eqref{unbalanced}. Before turning to the proof itself, let us make some remarks, that shall prove useful here and also in Section \ref{Oriented invasion}. In order to 
derive the invasion property in a given domain $\Omega$, we shall use ``sliding-type" arguments. Such arguments rely on the existence of compactly supported  generalized subsolutions of the stationary problem associated with \eqref{evol homogeneous}. The latters are given by solutions of 
the following Dirichlet problem:
\begin{equation}\label{tronque}
\left\{
\begin{array}{rll}
-\Delta u& = f(u), \quad &x\in B_{R} \\
u &= 0 , \quad &x \in \partial B_{R}.
\end{array}
\right.
\end{equation}
If $R>0$ is large enough, the existence of a positive solution of \eqref{tronque}
is classical. One could also reclaim the construction done in Section \ref{sec:persistence}
for the more general problem~\eqref{stationary}. 
In the following, such solutions are denoted by $u_{R}$. Thanks to the celebrated result of Gidas, Ni and Nirenberg \cite{GNN}, $u_{R}$ is radially symmetric and decreasing. 
In the sequel, we shall use functions of the form $u_{R}(\cdot - z)$ as stationary subsolutions of \eqref{evol homogeneous}. In this sense, $u_{R}$ will
be understood to be extended by $0$ outside 
$\ol B_{R}$. We also mention for future use that, thanks to Proposition \ref{pro:urto1},\begin{equation}\label{large}
\forall \eta \in (\theta , 1),\ \forall r>0, \ \exists R>0,\quad \min _{B_{r}}u_{R} > \eta.
\end{equation}

\begin{proof}[Proof of Theorem \ref{th invasion bistable}]
By the previous considerations, we can take $R>0$ large enough so that the function
$u_R$ is non-negative, radially symmetric and decreasing and satisfies 
$$M := \max u_R = u_R (0) > \theta.$$ 
Moreover, $M<1$ by Lemma \ref{lem:ur}. Let $K\subset \mathbb{R}^{N}$ be a
star-shaped compact set. Up to a coordinate change, we can assume that it is star-shaped
with respect to the~origin. Set
$$
L:= 2\diam(K)+2R+1,
$$
where $\diam(K)$ stands for the diameter of $K$. We then define 
$$
\Omega_{1} := (K + L\mathbb{Z}^{N})^{c}.
$$
In order to fall into our standing assumptions, we require that $K^c$ has a $C^3$ 
boundary.
Call $z:=(\frac L2,0,\dots,0)$.
Let $u$ be the solution of \eqref{evol homogeneous} on $\Omega_{1}$ with initial datum $u_{R}(\cdot - z)$. Notice that $u_{R}(\cdot - z)$ is compactly supported in $\Omega_{1}$,
whence it is a generalized stationary subsolution of \eqref{evol homogeneous}. Classical arguments show that $u(t,x)$ is increasing with respect to $t$ and converges locally uniformly in $x\in \ol\O_1$ to a stationary solution $u_{\infty} > u_R (\cdot - z)$ of \eqref{evol homogeneous}
as $t$ goes to $+\infty$. 
In particular,  $u_{\infty} >0$ in $\ol\O_1$
by the elliptic strong maximum principle and the Hopf lemma.
We claim that $u_{\infty} \equiv 1$. The proof of this fact is divided
into three steps.

\medskip
\emph{Step 1. Lower bound ``far from the boundary".} \\
Let us show that
\begin{equation}\label{exterior of balls}
   \forall y \in \left( \bigcup_{k\in L\mathbb{Z}^{N}}B_{\frac L2}(k) \right)^{c}, \quad   u_{\infty} > u_R (\cdot - y).
\end{equation}
Fix an arbitrary point $y$ in the set 
$\left( \bigcup_{k\in L\mathbb{Z}^{N}}B_{\frac L2}(k) \right)^{c}$. 
This set is closed, path-connected, contains the point $z$ and its distance from
$\O_1^c$ is larger than or equal to $L/2-\diam(K)>R$.
We can then find a continuous path 
$\gamma \, : \, [0,1]\mapsto \Omega_{1}$ such that
$$
\gamma(0) = z, \ \gamma(1) = y, \ \text{ and }\  \ol B_{R}(\gamma(s)) \subset \Omega_{1},\ \text{ for all } \ s\in [0,1].
$$
Then, for $s \in [0,1]$, define
$$
s^{\star} := \sup\left\{  s\in [0,1] \ : \ \forall s^{\prime} \in [0,s], \ u_{\infty}>u_{R}(\. - \gamma(s))  \right\}.
$$
Observe that the above set is non-empty because
$u_R (\cdot - z)<u_{\infty} $.
If we show that $s^{\star}=1$ then \eqref{exterior of balls} is proved.
Assume by contradiction that $s^{\star}<1$. Then, the fact that
$u_{R}(x - \gamma(s^\star))$ is continuous with respect to $s$ and $x$ and it is compactly supported
implies that $u_{R}(\. - \gamma(s^\star))$ touches $u_{\infty}$ from below, in the sense that
$\min(u_{\infty}-u_{R}(\. - \gamma(s^\star)))=0$. 
The point(s) at which the minimum is attained necessarily belongs to $B_{R}(\gamma(s^\star)) \subset \Omega_{1}$
because $u_{\infty}>0$ in $\ol\Omega_{1}$.
The elliptic strong  maximum principle 
eventually yields $u_{\infty}\equiv u_{R}(\. - \gamma(s^\star))$, which is impossible
again because $u_{\infty}>0$.  

\medskip
\emph{Step 2. Lower bound ``near the boundary".} \\
We show now that
\begin{equation}\label{interior of balls}
  \forall x \in\ol\O_1\cap
  \bigcup_{k\in L\mathbb{Z}^{N}}B_{\frac L2}(k) , \quad u_{\infty}(x) \geq M .
\end{equation}
To do this we use the same sliding method as before, but we need a refinement. Namely, we need the support of the subsolution to cross the boundary of  $\O_1$. This is where the star-shaped condition comes into play.

For $l>0$, define the function
$$
v_{l} := \max_{e\in \mathbb{S}^{N-1}}u_{R}(\cdot -l e).
$$
Recalling that $u_{R}$ is radially symmetric and decreasing, we see that $v_{l}$
is a radial function with respect to the origin which is increasing on
$\ol B_l\setminus
\ol B_{l-R}$ (with the convention $\ol B_{l-R}=\{0\}$ if $l\leq R$) and decreasing
on $\ol B_{l+R}\setminus \ol B_l$. 
We know from the previous step that $u_{\infty}>v_{L/2}$.
We define
$$
l^{\star} := \inf \left\{ l \geq0 \  : \  \forall l^{\prime} \in \left[l ,\frac L2 \right], \ u_{\infty} \geq v_{l^{\prime}}    \right\}.
$$ 
Assume by contradiction that $l^{\star} > 0$. By continuity, 
$v_{l^\star}$ touches $u_{\infty}$ from below at some point 
$\bar x\in\ol\O_1\cap (B_{l+R}\setminus \ol B_{l-R})$
(the contact cannot happen on $\partial B_{l+R}\cup\partial B_{l-R}$
because $v_{l^\star}=0$ there if $l\geq R$ and 
$v_{l^\star}$ cannot be touched from above at $0$ by a smooth function if
$l< R$).
We necessarily have that $|\bar x|\leq l^{\star}$,
because otherwise $v_{l}(\bar x)$ would be larger  than
$v_{l^\star}(\bar x)=u_\infty(\bar x)$ for $l<l^\star$ close to $l^\star$, contradicting the definition of $l^\star$.
It is easy to check that
$v_{l^\star}(\bar x)=u_{R}(\bar x -l^\star e)$, with $e=\bar x/|\bar x|$, whence
$u_{R}(\cdot -l^\star e)$ 
touches $u_{\infty}$ from below  at the point 
$\bar x$. 
If $\bar x \in \Omega_1$ then the elliptic strong  maximum principle implies that  
$u_{R}(\cdot -l^\star e)\equiv u_{\infty}$ on $\ol\O_1$, which is not the case
because $u_{R}(\cdot -l^\star e)$ is not  everywhere positive on $\ol\O_1$. Hence, 
$\bar x \in \partial K$. 
Recalling that $|\bar x|\leq l^{\star}$, we find that 
$\nabla u_{R}\left(\bar x-l^{\star}e \right)$ is positively collinear to $e$, whence

$$(l^\star e-\bar x)\cdot\nu(\bar x)=\bigg(\frac{l^\star}{|\bar x|} -\bar x\bigg)\,\bar x\.\nu(\bar x) \leq 0.$$

$$
\nu(\bar x)\.\nabla u_{R}(\bar x -l^\star e) = 
 \left( \nu(\bar x)\cdot e \right) \vert \nabla u_{R}\left(\bar x-l^{\star}e \right)\vert 
 \leq 0,$$
the latter term is non-positive because $K$ is star-shaped with respect to $0$. But 
then the Hopf lemma leads to
the same contradiction $u_{R}(\cdot -l^\star e)\equiv u_{\infty}$ 
as before. We have reached a contradiction.
This shows that $l^{\star} = 0$. In particular, we find that 
$u_\infty\geq M$ on $\ol\O_1\cap\ol B_{L/2}$.

The above argument can be repeated with $v_l$ replaced by $v_l(\.+k)$
for any $k\in L\mathbb{Z}^{N}$, leading to the property \eqref{interior of balls}.

\medskip
\emph{Step 3. Conclusion.} \\
We know from \eqref{exterior of balls}-\eqref{interior of balls} that $u_{\infty}\geq  M>\theta$
in the whole $\Omega_{1}$. Then, $u_\infty$ being an entire (stationary) solution of~\eqref{evol homogeneous},
Lemma~\ref{ODE} implies that $u_\infty\equiv 1$.
Thus, the invasion property holds for the initial datum $u_R(\cdot -z)$. We then~derive property~$(ii)$ of Theorem \ref{invasion general}
 by arguing as in the proof of Theorem \ref{th persistence} in Section~\ref{sec:persistence}.
 \end{proof}


\subsection{Blocking}\label{blocking}

This section is dedicated to the proof of Theorem \ref{th blocking bistable}. Namely,
we exhibit a smooth periodic domain $\Omega_{2}$ where the blocking property holds for
any compactly supported data $\leq1$. 
We use a result by Beresycki, Hamel and Matano, 
who construct a compact set $K \subset B_{1/2}$ such that there is a classical solution to the problem
\begin{equation}\label{syst block}
\begin{cases}
-\Delta w = f(w) & \text{ in }B_{1}\setminus K \\
\partial_{\nu}w = 0 & \text{ on }\partial K \\
w=1 &\text{ on } \partial B_{1}, \\
0<w<1& \text{ in }B_{1}\setminus K,
\end{cases}
\end{equation}
see \cite[Theorem 6.5 and (6.8)]{BHM}. The function $w$ will act as a \emph{barrier} which prevents invasion. Our domain $\O_2$ is depicted in Figure \ref{domain BHM}.


\begin{figure}[H]
\begin{center}
\includegraphics[scale = 0.5]{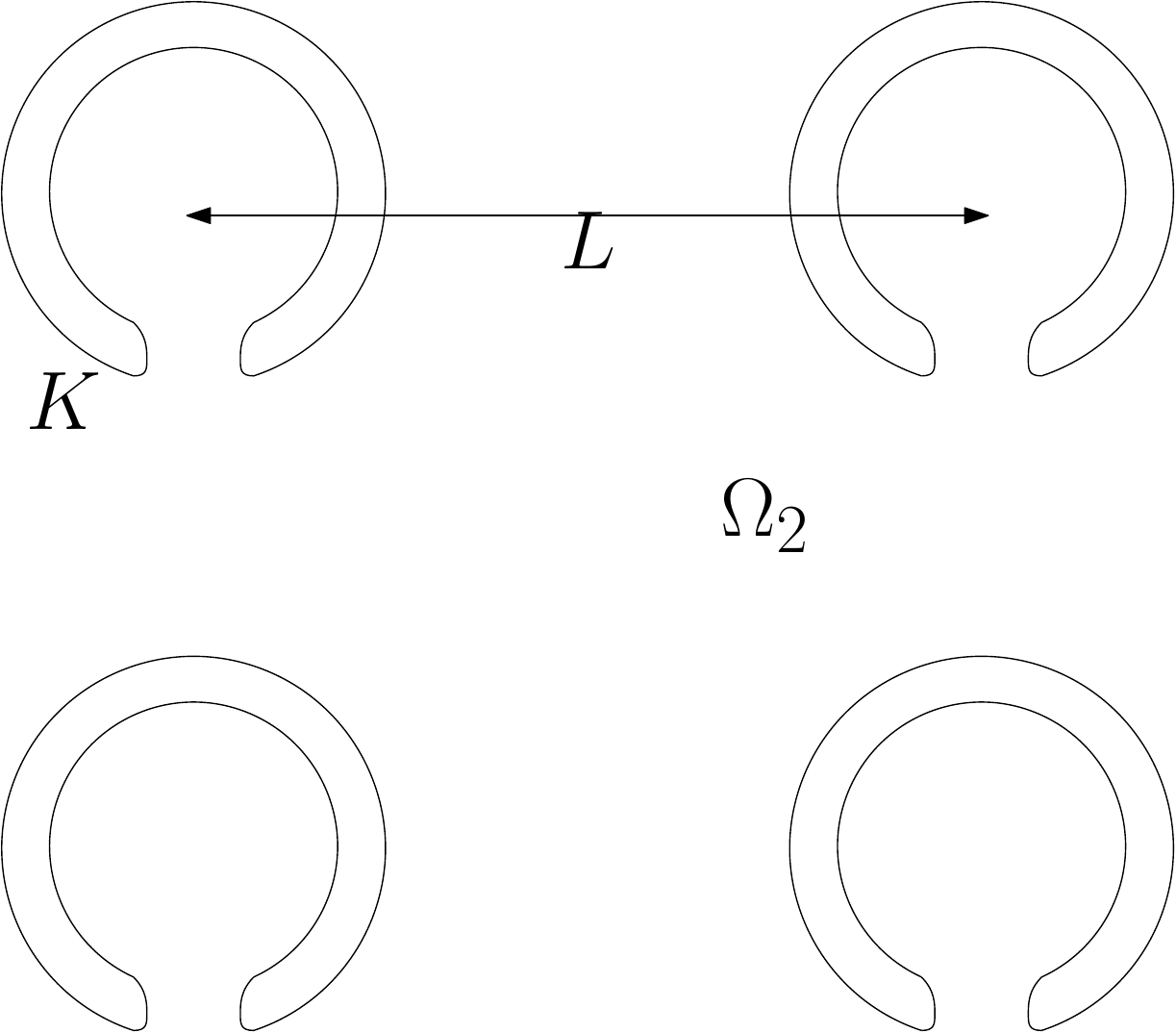}
 \caption{The domain $\Omega_{2}$}
    \label{domain BHM}
    \end{center}
    \vspace{-10pt}
\end{figure}

\begin{proof}[Proof of Theorem \ref{th blocking bistable}] 
	As at the beginning of the proof of Theorem \ref{th invasion bistable} in the previous section, we start
	with taking $R>0$ so that the solution $u_R$ to \eqref{tronque} is nonnegative and satisfies $\theta<\max u_{R}<1$. 
	Call $L:= 2R+2$ and define
	$$
	\Omega_{2} := (K + L\mathbb{Z}^{N})^{c}.
	$$
We prove the three statements of the theorem separately.

\medskip
\emph{Statement $(i)$.} \\
Let $u$ be the solution of \eqref{evol homogeneous} arising from a compactly supported initial datum $u_{0}\leq1$. Let $M>0$ be large enough so that $\supp (u_0) \subset B_{M}$. 
Consider the
solution $w$ of \eqref{syst block} given by \cite{BHM}. Define the function $\t w$ as follows:
$$
\forall k\in L\mathbb{Z}^{N}\ \text{such that }|k|>M+1,\
\forall x\in\ol\O_2\cap B_1,\quad
\tilde{w}(k+x) := w(x),$$
then extended to $1$ in the rest of $\ol\O_2$.
This is a generalized stationary supersolution of \eqref{evol homogeneous}. Hence, because
 $u_{0}\leq \tilde{w}$, the parabolic comparison principle yields that 
\begin{equation}\label{well}
\forall t \geq 0, \ \forall x \in \Omega_{2} , \quad  u(t,x)\leq \tilde{w}(x).
\end{equation}
This implies that $u$ does not satisfy the invasion property.
Let us see that $u$ is actually blocked in the sense of Definition \ref{def}. We argue by contradiction.
If this were not the case, we would be able to find a diverging sequence $(t_{n})_{n\in \mathbb{N}}$ and a sequence 
$(x_{n})_{n\in \mathbb{N}}$ in $\O_2$ such that $u(t_{n},x_{n}) \to 1$ as $n$ goes to $+\infty$. Then, defining
$$
u_{n} := u(\cdot +t_{n},\cdot +k_{n}),
$$
where $k_{n} \in L\Z^{N}$ is such that $x_{n} - k_{n} \in \mathcal{C}$, the parabolic estimates would
 allow us to extract  a subsequence of $(u_n)_{n\in\N}$ converging locally uniformly in $\ol\O_2$
  to some entire solution $u_{\infty}$  of \eqref{evol homogeneous} such that $\max   u_{\infty}=1$. 
  The parabolic comparison principle and the Hopf lemma would yield $u_{\infty} \equiv 1$. 
  Hence, from \eqref{well} we would get that $\t w(\cdot +k_{n})\to1$ as $n$ goes to $+\infty$ 
  (up to subsequences), locally   uniformly on $\ol\O_2$,
  which is impossible because
  $\inf_{B_{M+L+2}} \t w(\cdot +k)\leq \inf w<1$ for any $k\in L\Z^{N}$.
  
\medskip
\emph{Statement $(ii)$.} \\
Now,  let $\tilde{u}$ be the solution of \eqref{evol} emerging from the initial datum $u_R(\cdot - z)$, 
where $z:=(\frac L2,0,\dots,0)$. Observe that $\supp u_R(\cdot - z)=\ol B_R(z)\subset \O_2$, whence
$u_R(\cdot - z)$ (extended by $0$ outside $\ol B_R(z)$) is a stationary generalized subsolution of \eqref{evol homogeneous}. It follows that, as $t$ goes to $+\infty$, $\tilde{u}(t,x)$ converges increasingly to
 a stationary solution $p(x)$ of~\eqref{evol homogeneous} satisfying
 $u_R(\cdot - z)\leq p\leq1$. We have proved above that $p \not\equiv 1$. Let us show that $p$ is periodic.

Take $k \in L\mathbb{Z}^{N}$. We can find a continuous path $\gamma \, : \, [0,1] \mapsto \Omega_{2}$ such that 
$$
\gamma(0) = z, \ \gamma(1) = z+k, \text{ and }\ \ol B_{R}(\gamma(s)) \subset \Omega_{2}, \ \text{ for } s\in [0,1].
$$
Then, we can argue as in the proof of Theorem \ref{th invasion bistable} to get that
$$
p \geq u_{R}(\cdot - z -k).
$$
Using $u_{R}(\cdot - z -k)$ as an initial datum for \eqref{evol homogeneous}, the parabolic comparison principle yields
$$
p\geq p(\cdot-k).
$$
This being true for all $k\in L\mathbb{Z}^{N}$, we have that $p$ is indeed periodic, 
with the same periodicity as $\O_2$. 
Finally, the fact that $1>\inf p>u_R(0)>\theta$ readily implies that $p$ cannot be constant.

\medskip
\emph{Statement $(iii)$.} \\
Let us show that invasion fails for front-like initial data. Let $v_0$ be a front-like 
initial datum in a direction 
$e \in \mathbb{S}^{N-1}$, in the sense of~\eqref{def f l}. 
Consider again the solution $w$ of \eqref{syst block}, extended by $1$ outside $\ol B_1$. 
Then, there is $M>0$ such that
$$
\sup_{\su{x\cdot e >M}{x\in\O_2}} v_0(x) < \inf w.
$$
We take $k \in L\Z^{N} $ such that $k\cdot e > M+1$. It holds that 
$w(\cdot -k)\geq v_0$. Thus
the parabolic comparison principle yields that the solution of \eqref{evol homogeneous} arising from $v_0$ 
lies below $w(\cdot -k)$, whence cannot converge locally uniformly to $1$ as $t$ goes to $+\infty$. 

Suppose now that \eqref{evol homogeneous} admits a pulsating front solution $v$ (recall Definition \ref{def:ptf}).
On one hand, we have just shown that the invasion property fails for $v$, which implies that
its speed $c$ cannot be positive. 
On the other hand, statement~$(ii)$ provids us with a solution with a compactly supported initial datum 
$u_0<1$ which converges to a positive periodic steady state  as $t$ goes to $+\infty$. 
Up to translation, $u_0$ can be fit below $v(0,\cdot)$,
and thus we deduce by comparison that the speed $c$ cannot be $\leq0$ either.
Hence, pulsating fronts do not exist for problem \eqref{evol homogeneous} in $\O_2$
\end{proof}

\section{Oriented invasion}\label{Oriented invasion}

In this section, we construct some domains which exhibit a new phenomenon, that we call oriented invasion, 
which is between blocking and invasion. Namely, invasion occurs in a direction but is blocked in the opposite one. 

Throughout this section, the nonlinearity $f$ is of the unbalanced bistable type \eqref{unbalanced}. 
As in Section \ref{section bistable}, we shall use some ``sliding-type" arguments to prove that invasion occurs in some directions,
and some ``barriers" to get the blocking in other directions. Recall that these methods worked
under suitable
geometric conditions on the domain.
The whole issue here is to construct a periodic domain which roughly satisfies one type of condition
in some directions and the other type in other directions.

Let us make the geometric condition required in the sliding method explicit.
We will slide the same functions $u_{R}(\cdot-z)$ as in Section
\ref{section bistable}, extended by $0$ outside $\ol B_R(z)$ and restricted to $\O$,
which are generalized subsolutions of the first equation in \eqref{evol homogeneous}. We recall that
for $R$ large enough, $u_R$ is positive in $B_R$, with
$\theta<\max u_{R}<1$. By~\cite{GNN}, we further know that
$u_{R}$ is radially symmetric and decreasing on $\ol B_R$,  whence
$$
\forall x \in (\partial \O)\cap B_{R}(z), \quad \nu(x)\cdot \nabla u_{R}(x-z) = \nu(x) \cdot \frac{z-x}{\vert z-x \vert} \vert  \nabla u_{R}(x-z) \vert.
$$
It follows that $u_{R}(\cdot-z)$
 is a generalized stationary subsolution of \eqref{evol homogeneous} if and only if $\O$ fulfils the following geometric
 condition:
\begin{equation}\label{condition subsolution}
 \forall x \in (\partial\O)\cap \ol B_{R}(z), \quad   (z-x)\cdot\nu(x) \leq 0.
\end{equation}

\subsection{Oriented invasion in a periodic cylinder}\label{sec:cyl}

We start with showing that the oriented invasion occurs in cylindrical domains.
We consider problem \eqref{evol homogeneous} set in a 
{\em periodic cylindrical} domain, that is, of the~form
$$
\O= \left\{  (x_{1},x_{2})\in \mathbb{R}\times\mathbb{R}^{N-1} \ : \  x_2   \in \omega(x_1)    \right\},
$$
where $\omega(\.)\subset \R^{N-1}$ is such that $\O$ is of class $C^3$, 
connected and periodic in the sense that there is $L>0$ such that $\omega(\cdot +L) = \omega$.
Throughout this section, the points in $\O$ will be denoted by $(x_1,x_2)\in\R\times\R^{N-1}$, and $e_1 := (1,0 ,\ldots,0)$ will be the unit vector in the direction of the axis of the cylinder. 

\begin{theorem}\label{OI cylinder}
There exists a periodic cylindrical domain $\O$ and a positive constant~$c$ 
such that, for every $\eta>\theta$, there is $r>0$ for which 
the following properties hold for every solution to \eqref{evol homogeneous} arising from a
compactly supported initial datum satisfying 
$$0\leq u_0\leq 1,\qquad
u_0>\eta\quad \text{ in }\ \O\cap B_r.$$
\begin{enumerate}[$(i)$]
\item {\em Invasion to the right:}
$$
\forall 0<c_{1}<c_{2}<c, \quad \underset{{\substack{c_{1}t \leq x_1\leq c_{2} t \\  x_2  \in \omega(x_1)}}}{\min}  \big|u(t,x_1,x_2)-1\big| \underset{t \to +\infty}{\longrightarrow}0.
$$
\item {\em Blocking to the left:}
$$
u(t,x_1,x_2)  \underset{x_1 \to -\infty  }{\longrightarrow} 0
\quad\text{ uniformly in }\ t>0,\ x_2 \in \omega(x_1).
$$
\end{enumerate}
\end{theorem}


The blocking result makes use of the results by Berestycki, Bouhours and Chapuisat~\cite{BBC}, 
which are in turn
inspired by \cite[Theorem 6.5]{BHM} by 
Berestycki, Hamel and Matano, already used here in Section \ref{blocking}. 
In \cite{BBC}, the authors build an asymptotically straight cylinder 
for which all solutions initially confined in the half space $\{x_1 <0\}$ do not invade. 
The mechanism they exploit is that the propagation is hampered by the presence of a ``narrow passage''
which suddenly widens. 
More precisely, we shall need the following.
\begin{proposition}\label{prop bbc}
Let $\O$ be a periodic cylindrical domain and let $b\in \R$. There exists a positive constant $\varepsilon$ depending~on 
$$
\left\{ (x_1,x_2) \in \O \ : \ x_1 \in (b,b+1)  \right\}
$$
such that, if
$$
\left\vert \left\{ (x_1,x_2) \in \O \ : \ x_1\in (b+1,b+2)  \right\}  \right\vert < \varepsilon,
$$
then the following problem admits a positive solution:
\begin{equation}\label{eq block L}
\begin{cases}
-\Delta w = f(w), \quad &(x_1,x_2) \in \O,\ x_1< b+2 \\
\partial_{\nu} w = 0, \quad &(x_1,x_2) \in \partial \O, \ x_1<b+2  \\
w(b+2,x_2)= 1, \quad  & x_2 \in \omega(b+2) \\
w(x_1,x_2)   \underset{x_1 \to -\infty  }{\longrightarrow} 0 &
\text{uniformly in } x_2  \in \ol\omega(x_1).
\end{cases}
\end{equation}
\end{proposition}

This proposition can be extracted from the proof of \cite[Theorem 1.8]{BBC}. 
We shall not redo the proof here, but we  mention the ideas for reader's ease, assuming that $b=0$ for notational simplicity. They rely on the same energy method used here in Section \ref{sec:persistence}, with the energy functional defined on the truncated cylinders $\O \cap \{  -R < x_1 <0 \}$ with boundary conditions $w(-R,\cdot) = 0$, $w(0 , \cdot) = 1$. The idea is then to take the limit $R\to+\infty$, and try to get the last condition of \eqref{eq block L} at the limit. One needs to be cautious there: it is crucial to 
have that the selected minimizers do not converge to $1$ as $R$ goes to $+\infty$, which could happen if one takes global minimizers. The authors take instead local minimizers in a suitable energy well which converge to a solution of \eqref{eq block L}.

Proposition \ref{prop bbc} will allow us to derive the ``blocking'' property $(ii)$ of Theorem~\ref{OI cylinder}
by considering a periodic cylindrical domain $\O$ containing narrow passages which widen very suddenly in the leftward
direction. Conversely, for the ``invasion'' property~$(i)$, we need such passages to open slowly in the 
rightward direction; this will allow us to construct a front-like subsolution by ``bending'' the level sets 
of a planar front. The domain $\O$ is depicted 
in Figure \ref{fig:1}.



\begin{figure}[H]
\begin{center}
\includegraphics[scale = 0.47]{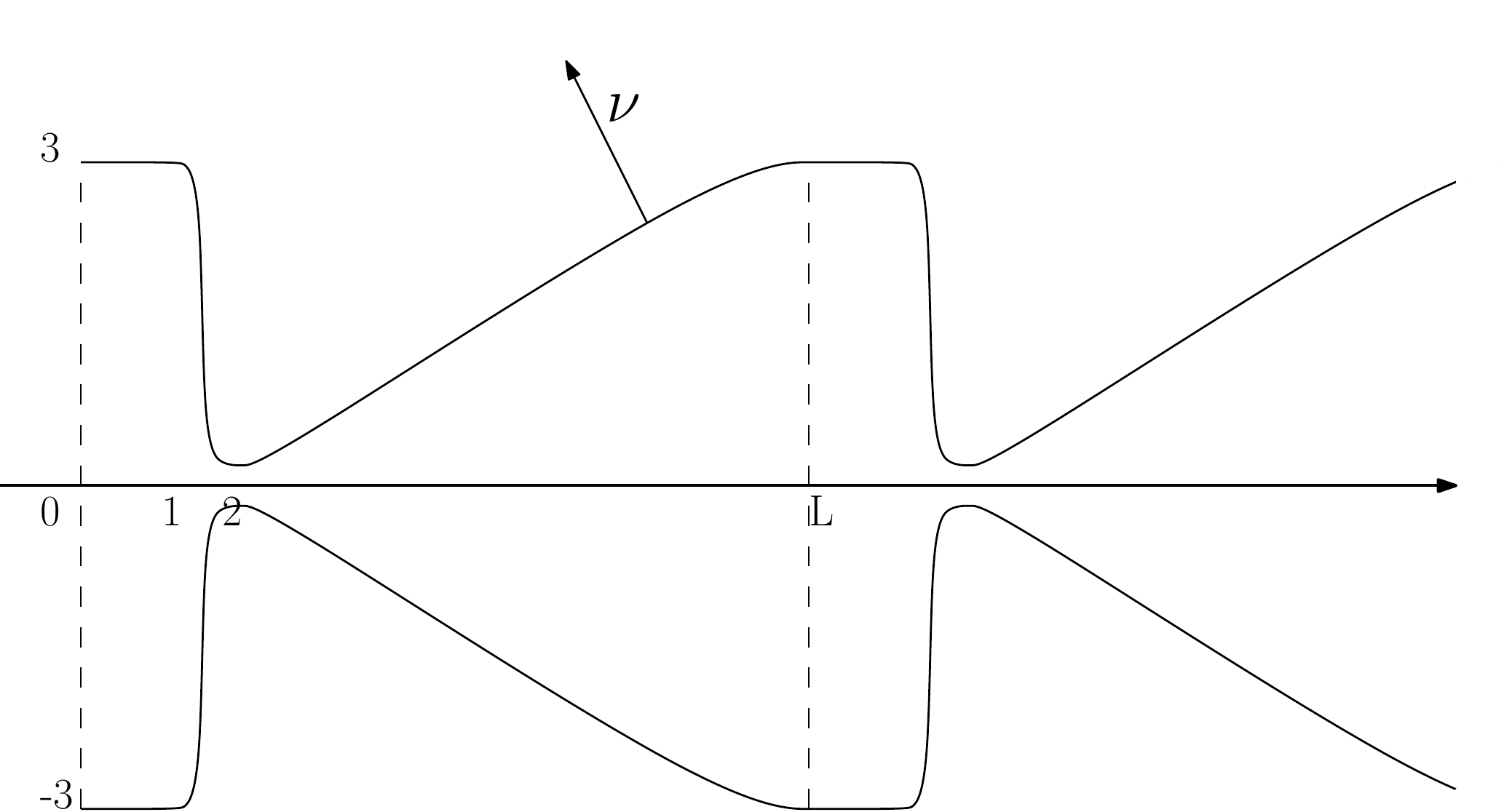}
 \caption{The periodic cylinder}
    \label{fig:1}
    \end{center}
\end{figure}

\begin{proof}[Proof of Theorem \ref{OI cylinder}]
	We define the domain $\O$ as follows:
	$$
	\Omega:= \left\{  (x_1,x_2)\in \mathbb{R}\times\mathbb{R}^{N-1} \ : \ \vert x_2 \vert  \leq \upsilon(x_1)    \right\},
	$$
	where $\upsilon:\R\to(0,+\infty)$ is a $C^\infty$ periodic function, with period $L$ to be chosen,
	satisfying 
	\begin{equation}\label{def omega}
	\begin{cases}
	\upsilon = 2 \quad &\text{ on } \ [0,1], \\
	\upsilon^{\prime} \leq 0 \quad  &\text{ on } \ [1,2], \\
	0 \leq \upsilon^{\prime} \leq \frac{2}{L-4} \quad  &\text{ on } \ [2,L], \\
	\upsilon(\frac{L}{2}) = 1.
	\end{cases}
	\end{equation}
	We further require that
	$$
	\vert \{ (x_1 , x_2 ) \ : \ x_1 \in [1,2], \ \vert x_2 \vert \leq \upsilon(x_1)     \}\vert \leq \e,
	$$
	with $\e$ given by Proposition \ref{prop bbc}, with $b=0$. Recall that $\e$ does not depend on $L$. 
	We have that $\min \upsilon = \upsilon(2)$. 
	We shall take $L$ very large so that the cylinder ``opens slowly" to the right, which will ensure property $(i)$.  
	
	The reminder of the proof is divided into three~steps. 

\medskip
\emph{Step 1. Building a subsolution. } \\
This step is dedicated to the construction of a subsolution to \eqref{evol homogeneous} moving rightward, 
that will be used to prove that invasion occurs in this direction. 
We consider a  perturbation $\t f$ of the nonlinearity $f$,
the latter being extended for convenience by $0$ to $(-\infty,0)$. 
Namely $\t f$ is Lipschitz-continuous, it satisfies, for some $\mu \in (0,1-\theta)$ small enough,
$$\t f(-\mu) = \t f(1-\mu) =0,\qquad
\t f<f \ \text{ in }\ (-\mu , \theta),\qquad
0<\t f\leq \frac1{1+\mu}f \ \text{ in }(\theta , 1-\mu),$$ 
and it is still positively unbalanced (between $-\mu$ and $1-\mu$):
$$\int_{-\mu}^{1-\mu}\t f(s)ds>0.$$
Let $\phi(x_1-c't)$ be the (unique up to shift) traveling front 
for the $1$-dimensional equation $\partial_tu=\partial_{11}u+\t f(u)$, connecting
$1-\mu$ to $-\mu$ with speed $c'>0$, provided by~\cite{AW}. Namely, the profile $\phi$ satisfies
$$
\phi^{\prime\prime} + c' \phi^{\prime} + \t f(\phi) =0\quad\text{in }\R,
$$
with $\phi(-\infty)=1-\mu$, $\phi(+\infty)=-\mu$.
We then take $c \in (0,c')$ and define 
$$
\psi(t,x_1,x_2) := \phi(x_1+\eta\vert x_2\vert^{2}  -ct ),
$$
with 
$$\eta := \min \bigg(\frac{c'-c}{2(N-1)} , \frac{\sqrt{\mu}}{4}\bigg).$$
Let us show that $\psi$ is a subsolution of \eqref{evol homogeneous} provided $L$ is sufficiently large. 
Direct computation gives us 
\[
\begin{split}
\partial_{t}\psi - \Delta \psi &=-c\phi'-\big(1 +4\eta|x_2|^2\big)\phi''-2\eta(N-1)\phi'
\\
&=\big(c'(1+4\eta^{2}\vert x_2\vert^{2})-c-2\eta(N-1) \big)\phi^{\prime} + 
(1+4\eta^{2}\vert x_2\vert^{2})\t f(\psi)\\
&\leq (c'-c-2\eta(N-1) )\phi^{\prime} +  (1+4\eta^{2}\vert x_2\vert^{2})\t f(\psi).
\end{split}
\]
The first term above is non-positive by the definition of $\eta$ and the negativity of~ $\phi'$.
Concerning the second term,  on one hand, if $\psi\leq \theta $ then 
$(1+4\eta^{2}\vert x_2\vert^{2})\t f(\psi) \leq f(\psi)$.  
On the other,
if $\psi\geq \theta$, since $\vert x_2\vert \leq \max \upsilon = 2$ and $\eta\leq\sqrt{\mu}/4$, we find that
$$
(1+4\eta^{2}\vert x_2 \vert^{2})\t f(\psi)\leq 
(1+\mu)\t f(\psi) \leq f(\psi).
$$
Therefore, $\psi$ is a subsolution of the first equation in \eqref{evol homogeneous}. Let us check the boundary condition. Observe that the unit exterior normal to $\O$ at $(x_1 , x_2) \in \partial \O$ is positively collinear to $(-\upsilon^{\prime}(x_1) , \frac{x_2}{\vert x_2 \vert})$. Recalling that 
$\upsilon'\leq \frac{2}{L-4}$
and that
$|x_2|\geq\min\upsilon = \upsilon (2)$, we see that, for $t\in \R$ and $(x_1,x_2) \in \partial \O$,
\[
\begin{split}
\left\vert  \left(-\upsilon^{\prime}(x_1) , \frac{x_2}{\vert x_2 \vert}\right) \right\vert \partial_{\nu}\psi(t,x_1,x_2)  &= \big(-\upsilon^{\prime}(x_1)+2\eta \vert x_2 \vert\big)\phi^{\prime} \\
&\leq \left(-\frac{2}{L-4}+2\eta\upsilon(2)\right)\phi^{\prime}.
\end{split}\]
Hence, $\psi$  is indeed a subsolution of \eqref{evol homogeneous} 
for $L$ sufficiently large.

\medskip
\emph{Step 2. Invasion to the right.} \\
Take $R>0$ large enough so that there is a positive solution $u_{R}$ of \eqref{tronque}. Owing to~\eqref{large}, we can increase $R$ to have
\begin{equation}\label{square R}
\forall x_2 \in [-2,2],\quad
u_{R}(0,x_{2}) \geq 1-\mu.
\end{equation}
Let $u$ denote the solution of \eqref{evol homogeneous} with initial datum $u_R(\cdot -\frac{L}{2} e_1)$, 
extended by $0$ and restricted to $\Omega$. 
If $L>2R+4$, we have that $\upsilon'\geq 0$ in $(\frac{L}{2}-R,\frac{L}{2}+R)$, 
that is, the support of the initial datum does not touch any ``narrow passage".
Let us check that $u_R(\cdot - \frac{L}{2}e_1)$ is a generalized stationary subsolution of \eqref{evol homogeneous}. 
We just need to verify the boundary condition. We have seen at the beginning of Section \ref{Oriented invasion}
that this in turn reduces
to check condition \eqref{condition subsolution}.
For  $(x_1,x_2) \in \partial \O\cap \ol B_R(\frac{L}{2}e_1)$, 
recalling that $\nu(x)$ is positively collinear to $(-\upsilon^{\prime}(x_1) , \frac{x_2}{\vert x_2 \vert})$, with $|x_2|=\upsilon(x_1)$,
we find that
\[
\begin{split}
\left\vert  \left(-\upsilon^{\prime}(x_1) , \frac{x_2}{\vert x_2 \vert}\right) \right\vert\left(\frac{L}{2}-x_1,-x_2\right)\.\nu(x) 
&= \upsilon^{\prime}(x_1)\left(x_1 - \frac{L}{2}\right) - \upsilon(x_{1}) \\
&    \leq \frac{2}{L-4}R -  \upsilon(x_{1}).
\end{split}
\]
Now, because $(x_1,x_2) \in \ol B_R(\frac{L}{2}e_1)$, we deduce
 that $\upsilon(x_1)\geq \upsilon(\frac{L}{2}) -  \frac{2}{L-4}R= 1-  \frac{2}{L-4}R$, and therefore
 $$
\left\vert  \left(-\upsilon^{\prime}(x_1) , \frac{x_2}{\vert x_2 \vert}\right) \right\vert\left(\frac{L}{2}e_1-x\right)\.\nu(x) \leq \frac{4}{L-4}R -  1.
$$
Then, for $L$ sufficiently large, condition \eqref{condition subsolution} holds and thus $u_R (\cdot - \frac{L}{2}e_1)$ is a generalized stationary subsolution of \eqref{evol homogeneous}. As a consequence, $u(t,\cdot,\cdot)$ is increasing with respect to $t$ and, because of \eqref{square R}, there holds
 $$
  \forall t \geq 0, \ \forall |x_2|\leq\upsilon\Big(\frac L2\Big) ,
   \quad u\Big(t,\frac L2,x_2\Big) \geq 1-\mu.
 $$
 We can now fix $L$ large enough so that the the above property holds and that the
 function $\psi$ defined in the first step is a subsolution.
Define 
$$
\O^{+} := \left\{ (x_1,x_2)\in \O \ : \ x_1\geq \frac{L}{2}\right\}.
$$
We have that $\psi(t,\frac{L}{2},x_2)  \leq u(t,\frac{L}{2},x_2) $, for all $t\geq0, \ \vert x_2 \vert \leq \upsilon(\frac{L}{2})$. Moreover, up to translation in time of $\psi$, we can assume that
$$
 \forall (x_1,x_2) \in \O^{+}, \quad \psi(0,x_1,x_2)\leq u(0,x_1,x_2).
$$
Hence, the parabolic comparison principle implies that $\psi(t,x_1,x_2)\leq u(t,x_1,x_2)$ for $t\geq 0,\ (x_1,x_2) \in \O^{+}$. 
Recalling that $\psi(t,x_1,x_2) := \phi(x_1+\eta\vert x_2\vert^{2}  -ct )$, we derive
$$
\forall \gamma \in (0,c),\quad
\lim_{t\to +\infty}\left( \min_{\substack{ 0\leq x_1\leq \gamma t  \\   \vert x_{2} \vert \leq \upsilon(x_{1})   }}  u(t,x_1,x_2)\right) \geq 1-\mu.
$$
Take now $0<c_1<c_2 <c$ and consider a diverging sequence of times $(t^n)_{n\in\N}$ and a sequence $(x_1^n)_{n\in\N}$ in $\R$ such that $c_1t^n\leq x_1^n\leq c_2 t^n$. Consider the sequence $(k_n)_{n\in \mathbb{N}}$ in $\Z$  for which
	$0\leq x_1^n -k_n L<L $. Then, $k_n \to +\infty$ as $n$ goes to $+\infty$. 	We define the translated functions	
$$
u^{n} := u(\cdot+t^n,\cdot+k_n L, \cdot).
$$
The sequence $(u^{n})_{n\in \mathbb{N}}$ converges as $n$ goes to $+\infty$ (up to extraction)
	locally uniformly in $t\in\R$, $(x,y) \in \ol\O$,
	to a function
	$ u^{\infty}$ which is an entire solution of \eqref{evol homogeneous}. Now, take $\gamma \in (c_2 , c)$. For $t\in \mathbb{R}$ and $(x,y)\in \O$, we have, for $n$ large enough
	$$
	0<x + k_n L \leq x + c_2 t^n <\gamma (t +  t^n),
	$$
	from which we deduce
$$
\forall t\in \mathbb{R}, \ \forall (x,y) \in \O , \quad  u^{\infty}(t,x) \geq 1-\mu>\theta.
$$
Owing to Lemma \ref{ODE}, we see that $u^{\infty} \equiv 1$, i.e., $u^n$ converges to $1$ locally uniformly in $(x,y) \in \ol\O$, as $t$ goes to $+\infty$. Then, statement $(i)$ of the theorem holds for the solution $u$ with initial datum $u_R(\cdot - \frac{L}{2})$. One then recovers the class of initial data stated in the theorem
by arguing as in the proof of Theorem \ref{th persistence} in Section \ref{sec:persistence}.

\medskip
\emph{Step 4. Blocking.} \\
We make use the result of \cite{BBC}, Proposition \ref{prop bbc} above. Let $w$ be the function given by Proposition \ref{prop bbc} applied to $\O$ with $b = 0$ (we recall that we chose $\e$ so that this was possible). We extend $w$ by setting $w(x_1,\cdot)=1$ for $x_1 \geq 2$. Then, $w$ is a generalized supersolution of \eqref{evol homogeneous}, and so is $w(\cdot+kL,\cdot)$, where $k\in \mathbb{Z}$ and $L$ is the period of our cylinder. Then, if $u_{0}\leq1$ is compactly supported, we can find $k\in \mathbb{Z}$ such that $u_{0} \leq w(\cdot + kL , \cdot)$. The parabolic comparison principle yields that $u(t,x_1,x_2)\leq w(x_1+kL,x_2)$, for all $t>0, \ (x_1,x_2)\in \O$. 
The blocking property for $u$ then follows from the last condition in 
\eqref{eq block L}.
\end{proof}

\subsection{Oriented invasion in a periodic domain}

This section is dedicated to the proof of Theorem \ref{th oriented}. It is more involved than the proof of Theorem \ref{OI cylinder}, and we shall proceed in several steps. First, in Section~\ref{design domain}, we design the periodic domain $\Omega_{3} \subset \mathbb{R}^{2}$. Then, in Section \ref{preliminary lemmas}, we state some auxiliary lemmas, used in Section \ref{proof oriented invasion} to prove the oriented invasion property. In the whole section, we call $e_{1} := (1,0)$ and $e_{2} := (0,1)$ the unit vectors of the canonical basis.

\subsubsection{Designing the domain $\Omega_{3}$}\label{design domain}

We shall take $L_1, L_2>0$ large enough and a compact set $K \subset \mathbb{R}^{2}$, and we define 
\begin{equation}\label{definition domain}
\Omega_{3}:= \Big( K +L_{1}\mathbb{Z}\times L_{2}\mathbb{Z}\Big)^{c}.
\end{equation}
The domain $\Omega_{3}$ we have in mind is depicted in Figure \ref{fig:2}. 

\begin{figure}[H]
\begin{center}
\includegraphics[scale=0.65]{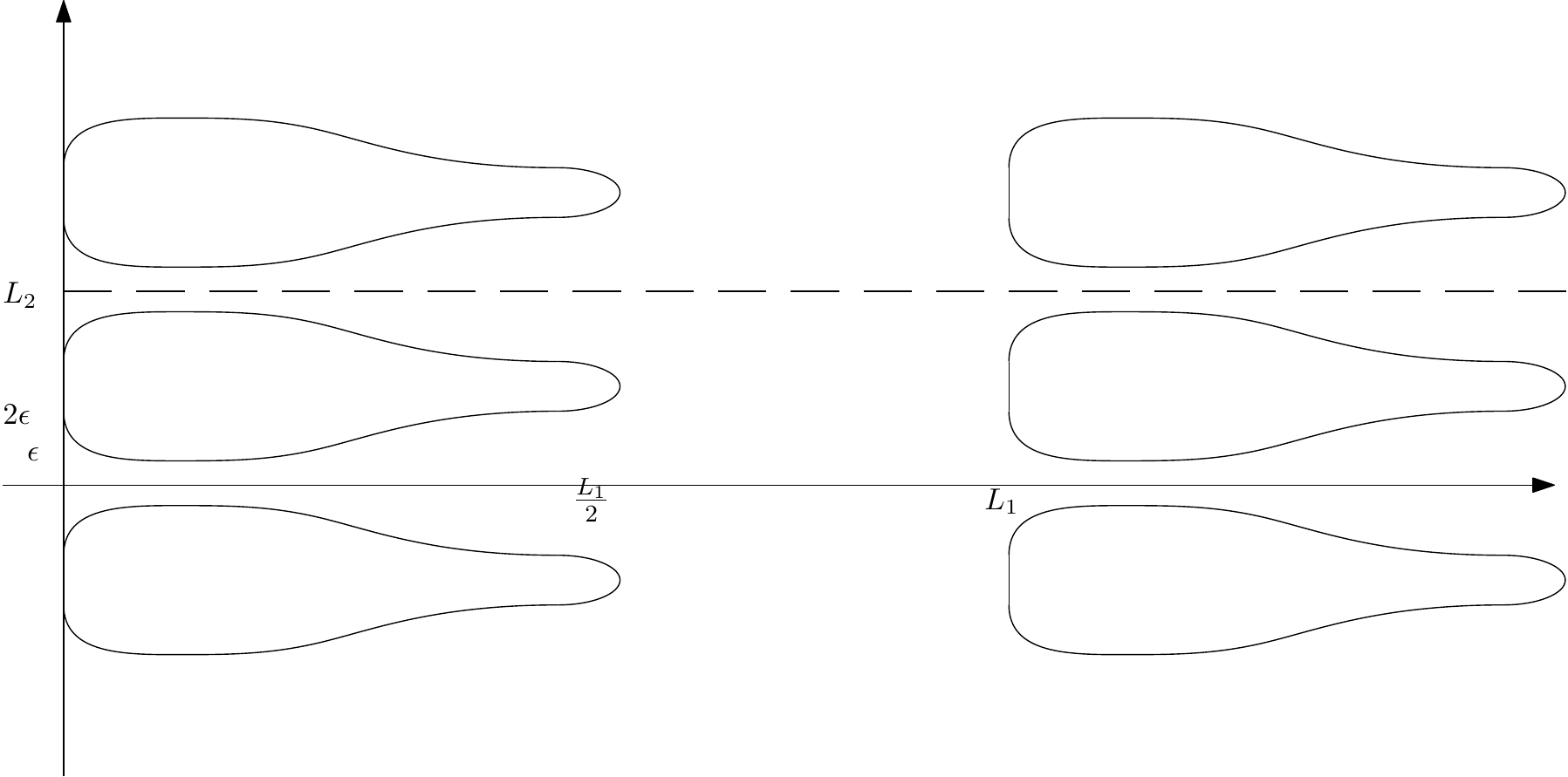}	
 \caption{The periodic domain $\Omega_{3}$}
    \label{fig:2}
    \end{center}
\end{figure}

The idea is to choose this domain in such a way that the narrow passages open in an abrupt way to the left, so that propagation will be blocked in this direction, but gently to the right, so that the solution will be able to pass, as in the case of the periodic cylinder 
of Section~\ref{sec:cyl}.

We shall build a function $h\in C([0,\frac{L_{1}}{2}])\cap C^\infty((0,\frac{L_{1}}{2}))$ to parametrize the boundary of $K$. More specifically, we define $K \subset [0, L_{1})\times[0,L_{2})$ by
$$K\cap(\{0\}\times\R)=\{0\}\times[2\varepsilon,L_{2}-2\varepsilon],$$
and
	$$\partial K \cap\left(\left(0,\frac{L_1}{2}\right]\times\left[0,\frac{L_{2}}{2}\right]\right)=\left\{(s,h(s))\ :\ s\in\left[0,\frac{L_1}{2}\right]\right\},$$
then reflected by symmetry with respect to the line $\{x\.e_2=L_2/2\}$:	
	$$\partial K \cap\left(\left(0,\frac{L_1}{2}\right]\times\left[\frac{L_{2}}{2},L_{2}\right]\right)=\left\{(s,L_{2}-h(s))\ :\ s\in\left[0,\frac{L_1}{2}\right]\right\}.$$
For $\O_3$ to be smooth, we need $h^{(n)}(0^+)=-\infty$ and $h^{(n)}(\frac{L_1}{2}^{-})=+\infty$, for any $n\geq1$.

Let $\e\in (0,1)$, $\kappa >0$, to be chosen later. We define $h$ as follows: first, on $[0,\frac{3}{\e^2}]$, we set
\begin{equation}\label{hyp h}
\begin{cases}
 h(0) =2\varepsilon,  \\
 h'(s) \leq 0,  \quad &s\in(0,1], \\
  h(s) =\varepsilon,    \quad &s \in [1,\e^{-2}], \\
0 \leq h'(s) \leq \e^2,   \quad &s \in [\e^{-2}, 2\e^{-2}],\\
h(s) =1,  \quad &s\in [2\e^{-2},3\e^{-2}].
\end{cases}
\end{equation}
%
%
Now, to define $h$ on $[\frac{3}{\e^2}, \frac{L_1}{2}]$, we introduce the following 
cut-off function $\chi \in C^{\infty}([0,1))$ such that
\begin{equation}\label{hyp khi}
\begin{cases}
\chi^{(n)}(0) =0,  \quad &n \geq 0, \\
\chi^{\prime}(s) \geq 0, \quad &s \in [0,1), \\
\chi(1^-) = 1, \ \chi^{(n)}(1^-) =+\infty,  \quad &n \geq 1,
\end{cases}
\end{equation}
%
and we set
$$
h(s) = 1+ \kappa \chi\left(\frac{1}{\kappa}\left(s - \frac{3}{\e^2}\right)\right) \quad \text{ for }\ s \in \left[\frac{3}{\e^2},\frac{3}{\e^2}+ \kappa\right].
$$
Finally, we define
$$
L_1 := 2\left(\frac{3}{\e^2} + \kappa\right) \ \text{ and } \ L_2 := 2(1+\kappa).
$$
The graph of the function $h$ is depicted in Figure \ref{fig:h}.
\begin{figure}[H]
\begin{center}
\includegraphics[height=6.5cm]{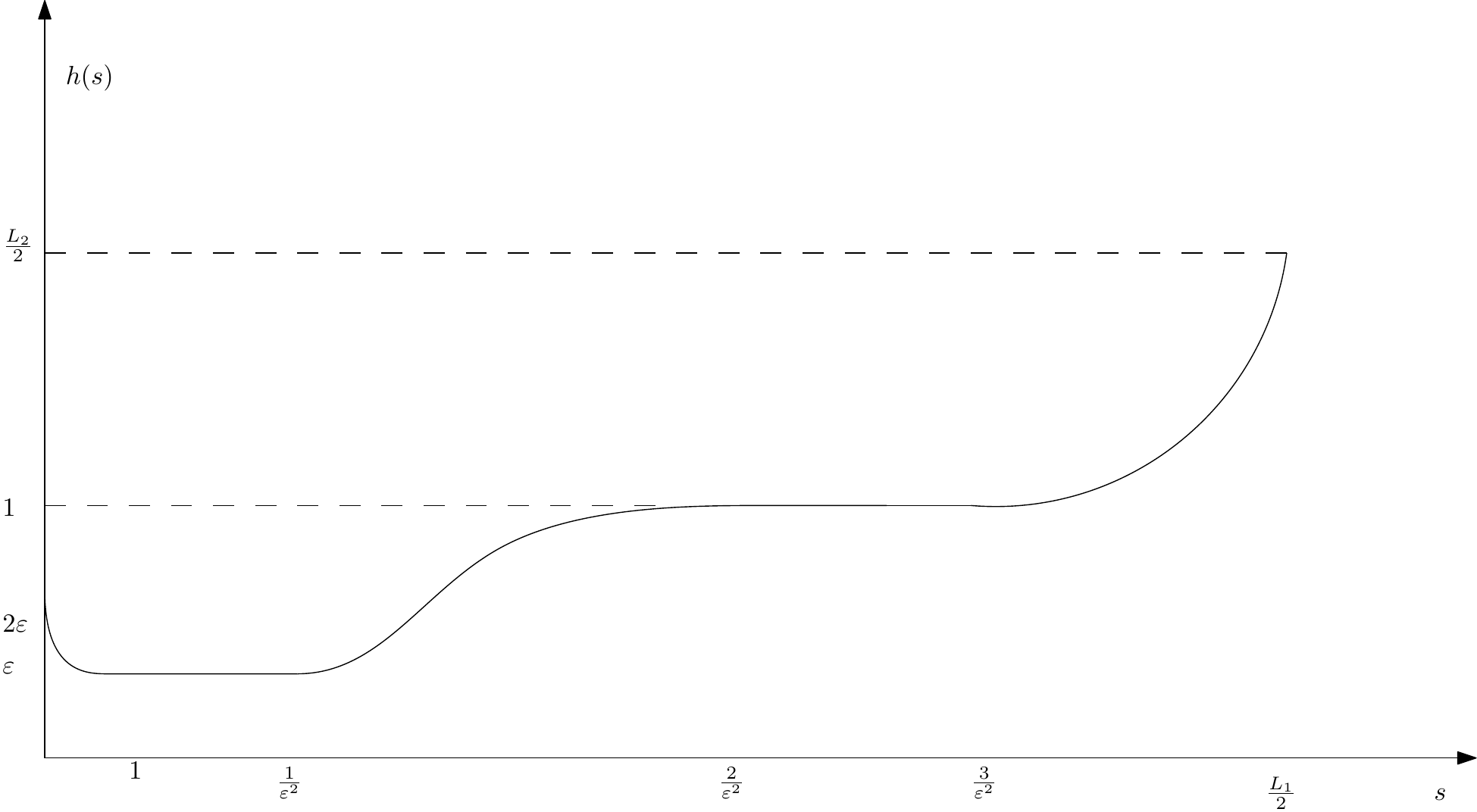}	
 \caption{Graph of the function $h$}
    \label{fig:h}
    \end{center}
\end{figure}   
We have defined the domain $\O_3$, depending on $\e, \kappa$. Let us see how we choose these parameters. To 
start with, we take $R>0$ large enough so that there is a positive solution $u_{R}$ of \eqref{tronque}. Thanks to \eqref{large} we can choose
$R>1$
 in such a way~that
\begin{equation}\label{inner estimate}
M := \min_{x \in \ol B_{1}}u_{R}(x) \in(\theta,1).
\end{equation}
Next, we take $\kappa$ large enough and $\e$ small enough so that $\O_3$ satisfies the following exterior ball condition at  every $x\in (\partial \O)\cap (  [\frac{2}{\e^2}, \frac{L_1}{2}] \times \R)$:
\begin{equation}\label{h 1}
\exists y \in \O_3^c \ \text{ such that } \  \ol B_{R^{2}}(y)\cap \ol \O_3 = \{ x \}.
\end{equation}
(Observe that $\kappa$ acts as a sialation in the definition of $h$.)
Now, Proposition \ref{prop bbc} applied to the periodic cylinder
$$
\tilde{\Omega}_{3} :=  \O_3\cap(\R\times (0 , L_{2}))
$$
and $b = -1$, yields that, if the measure of $\tilde\Omega_3\cap\big((0,1)\times\R\big)$
is small enough, then problem \eqref{eq block L} admits a solution in the truncated cylinder 
${\Omega}_{3}\cap\big((-\infty,1)\times(0,L_2)\big)$. Since this measure is smaller than $4\e$,
up to decreasing $\e$ we assume that this condition is fulfilled. 
We further increase $\kappa$ and decrease $\varepsilon$ so that
\begin{equation}\label{h prime}
\kappa\geq 4R \quad\text{and}\quad R \leq \frac{1}{2\e}.
\end{equation}

\subsubsection{Invasion towards right}\label{preliminary lemmas}

In this subsection, the domain $\O_3$ and the constant $R>1$ are the ones constructed before, and
$u_{R}$ is the solution to \eqref{tronque} extended by $0$ outside $\ol B_R$. We recall that it is radially symmetric and decreasing
and satisfies \eqref{inner estimate}.
We let $\mc{C}$ denote the following periodicity cell:
$$
\mathcal{C} := [2R,L_{1}+2R]\times  \left[-\frac{L_{2}}{2},\frac{L_{2}}{2}\right].
$$

Here is the key result to prove statement $(i)$ of Theorem \ref{th oriented}.
\begin{proposition}\label{prop oriented invasion}
Let $u$ be the solution of \eqref{evol homogeneous}
emerging from the initial datum~$u_{R}(\.-2Re_1)$. Then, there is $T>0$ such that, for any $n \in \mathbb{N}$, there holds
$$
u(t,x) > M>\theta \quad \text{ for }\  t \geq (n+1)T,\ x \in \ol\Omega_{3}\cap  
\bigcup_{\substack{a\in\N\cup\{0\},\ b\in\Z  \\ 
		a+|b|\leq n}} \left( \mathcal{C}+\{aL_{1}e_{1}+bL_{2}e_{2}\} \right).$$
\end{proposition}

The proof of this proposition is achieved through a series of intermediate lemmas.
The first two concern some geometric properties of $\Omega_{3}$.

\begin{lemma}\label{sliding1}
Let $\lambda \in [2R,3\e^{-2} - R]\cup [L_1-R,L_1]$. Then property \eqref{condition subsolution} holds with $z=\lambda e_{1}$,
and therefore $u_{R}(\cdot - \lambda e_{1})$  is a generalized subsolution
to \eqref{evol homogeneous}.
\end{lemma}

\begin{proof}
	Observe preliminarily that $2R<3\e^{-2} - R$ because $R<\e^{-1}<\e^{-2}$ by \eqref{h prime}. 
	We have already seen
	at the beginning of Section \ref{Oriented invasion} that property
	\eqref{condition subsolution} is equivalent to have that $u_{R}(\cdot - z)$ (restricted to $\ol\O_3$)
	is a generalized subsolution of \eqref{evol homogeneous}. 	
	
Take $\lambda \in [2R,3\e^{-2} - R]$ and $x\in ( \partial \Omega_{3})\cap \ol B_{R}(\lambda e_{1})$. 
Then $1<R \leq x\cdot e_{1} \leq 3\e^{-2} $ and $x\.e_2=\pm h(x\.e_1)$.
By symmetry, we can restrict to the case $x\.e_2=h(x\.e_1)$.
For such values of $x$ we have that $h^{\prime}(x \cdot e_{1}) \leq \e^2$. 
Because the unit exterior normal at the point $x\in \partial \O_3$, $\nu(x)$, is positively collinear to  
$( -h^{\prime}(x\cdot e_1) , 1)$, there holds
\begin{equation*}
\begin{array}{ll}
\big\vert ( -h^{\prime}(x\cdot e_1) , 1) \big\vert (\lambda e_{1}-x)\cdot \nu(x) &= -h^{\prime}(x\cdot e_{1})(\lambda-x\cdot e_{1})-h(x\cdot e_{1}) \\
&\leq \e^2\vert x\cdot e_{1} - \lambda\vert  -\varepsilon \\
&\leq \e^2R -\varepsilon \\
&\leq 0,
\end{array}
\end{equation*}
where the last inequality comes from \eqref{h prime}. 

Let us check property \eqref{condition subsolution} with $z= \lambda e_{1}$ and 
$\lambda \in  [L_1-R,L_1]$. Thanks to the periodicity of $\O_3$, we reduce to the case $\lambda \in  [-R,0]$.
Take $x\in(\partial \O_3) \cap \ol B_R ( \lambda e_1)$. Hence $x\.e_1 \in  [-2R,R]$ and thus in particular
$x\.e_1>-\frac{L_1}{2}$ thanks to \eqref{h prime}. It follows that $x\.e_1\geq0$, whence $x\.e_1\geq\lambda$.
Hence, by the same computation as above, using the fact that $h$ is decreasing on $[0,R]$ we derive \eqref{condition subsolution}.
\end{proof}

 \begin{lemma}\label{sliding2}
 Let $z \in \Omega_{3}$ be such that $z\cdot e_{1} \in [3\e^{-2} - R , L_{1}-2R]$ and $d(z,\partial \Omega_{3}) \geq 1$,
 where $d(z,\partial\Omega_{3})$ denotes the distance between $z$ and $\partial\Omega_{3}$.
  Then property~\eqref{condition subsolution} holds 
 and therefore $u_{R}(\cdot - z)$  is a generalized subsolution
 to \eqref{evol homogeneous}.
 \end{lemma}  
   
\begin{proof}
Take $z$ as in the statement of the lemma. Assume that there exists 
$x \in (\partial \O_3)\cap \ol B_{R}(z) $. We have that $L_1>x\cdot e_1 \geq 3\e^{-2}-2R\geq2\e^{-2}$ 
by \eqref{h prime} and therefore property \eqref{h 1} holds.
We see that $\vert x -y \vert = R^{2}$, $\vert x-z\vert < R$ and $\vert z -y\vert \geq R^{2} +1$. 
From this we get
\[\begin{split}
(z-x)\cdot \nu(x) &= \frac{(z-x)\cdot (y-x)}{\vert y-x \vert} \\
&= \frac{\vert z-x\vert^{2} + \vert y-x \vert^{2}-\vert z -y \vert^{2}}{2\vert y-x \vert} \\
&\leq  \frac{R^{2} + R^{4}-(R^{2}+1)^{2}}{2R^2}  \\
&< 0,
\end{split}\]
whence property~\eqref{condition subsolution}.
\end{proof}  
Now, for the sake of clarity, we state a lemma that gathers the previous two 
under a more useful form.
\begin{lemma}\label{slidable}
The set $\mathcal{S}$ defined by
\[\begin{split}
\mathcal{S} := &([2R, L_1]\times\{0\})\\
&\cup \{ z \in \O_3\cap\mc{C} \ : \ z\cdot e_1 \in [3\e^{-2} - R , L_1 - 2R] \ \text{ and } \ d(z , \partial \O_3) \geq 1   \}\\ & \cup 
\{  z\in (\ol\O_3\setminus B_{3R}(L_1 e_1))\cap\mc{C}
 \ : \ z \cdot e_1 \in [L_1 -2R , L_1]   \}
\end{split}\]
satisfies the following properties:
\begin{enumerate}[$(i)$]
\item The set $\mathcal{S}$ is a path-connected subset of $\O_3$.
\item  For $z\in  \mathcal{S}$, $u_{R}(\cdot - z)$ is a generalized stationary subsolution of \eqref{evol homogeneous}. 
\end{enumerate}
\end{lemma}
The set $\mathcal{S}$, depicted in Figure \ref{slidable set},
 is a ``slidable" region in $\O_3\cap \mathcal{C}$ for the centers of the subsolutions $u_{R}(\cdot - z)$. 
 Property $(ii)$ follows from 
 Lemmas~\ref{sliding1}-\ref{sliding2}, except 
when~$z$ belongs to the third set in the definition of $\mathcal{S}$.
But in such case condition~\eqref{condition subsolution} 
is readily derived by noticing that, thanks to \eqref{h prime},
the set $(\partial\O)\cap \ol B_{R}(z)$ is contained in the vertical line $\{x\.e_1=L_1\}$.
\begin{figure}[H]
	\begin{center}
		\includegraphics[scale = 0.6]{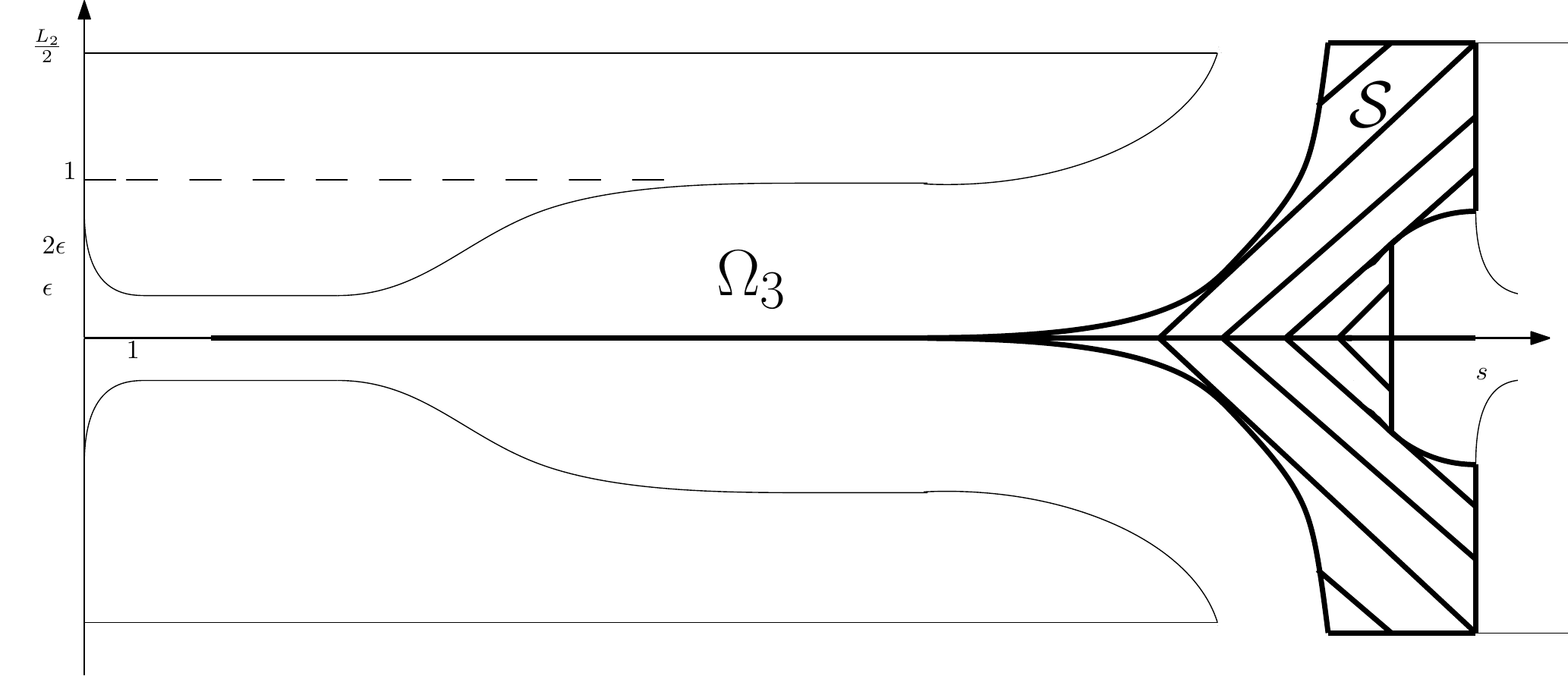}
		\caption{The ``slidable" region $\mathcal{S}$}
		\label{slidable set}
	\end{center}
	\vspace{-10pt}
\end{figure}  

The last auxiliary result, which will allow us to ``jump'' to the right of the narrow passage,
concerns the solution of the following Cauchy problem in $B_{R'}$, with $R^{\prime}>R$:
\begin{equation}\label{bornesup}
\left\{
\begin{array}{rll}
\partial_{t}v-\Delta v& = f(v), \quad &t >0 ,\ x\in B_{R^{\prime}} \\
v &= 0 , \quad &t>0 , \ x \in \partial B_{R^{\prime}} \\
v(0,x) &= u_{R}(x) , \quad  &x\in B_{R^{\prime}} .
\end{array}
\right.
\end{equation}
\begin{lemma}\label{lem:vR}
	Let $v_{R^{\prime}}$ be the solution of \eqref{bornesup} with $R'>R$.
	Then, for any $t> 0$, the function $v_{R^{\prime}}(t,\cdot)$ is radially symmetric and decreasing.
	Moreover, it satisfies
	$$
	\forall x \in B_{R^{\prime}-R},\quad
	\lim_{t \to +\infty}v_{R^{\prime}}(t,x) > M^{\prime}
	:= u_{R}(0).$$
\end{lemma}

\begin{proof}	
	The symmetry property is immediately inherited from the initial datum $u_R$ due to the 
	uniqueness of solutions for the parabolic problem \eqref{bornesup}.
	We show that the same is true for the radial monotonicity using a standard moving plane technique.
	Let $H$  be a straight line in $\mathbb{R}^{2}$ intersecting $B_{R^{\prime}}(0)$ which does not contain
	the origin and let $\Sigma$ denote the orthogonal symmetry
	with respect to $H$. We define 
	$$
	w_{R^{\prime}}(t,x) := v_{R^{\prime}}(t,\Sigma(x)), \quad \text{ for } \ t > 0, \ x \in \Sigma(B_{R^{\prime}}).
	$$
	Then, $w_{R^{\prime}}$ is a solution of \eqref{bornesup} set on $\Sigma(B_{R^{\prime}})$ arising from the initial datum $u_{R^{\prime}} (\Sigma(\cdot))$. Now, consider the domain
	$U$ given by the intersection between $\Sigma( B_{R^{\prime}})$ and the half-plane
	bounded by $H$ and containing the origin. Observe that $U\subset B_{R'}$.
	On the one hand,   $u_{R^{\prime}}$ and $v_{R^{\prime}}$
	coincide on $H\cap\ol B_{R'}$ for all $t$, because $\Sigma$ is the identity there.
	On the other hand, $u_{R^{\prime}}>v_{R^{\prime}}=0$ on
	$\partial U\setminus H$.
	Moreover, for $x \in U$ there holds $d(x,0)<d(x,\Sigma(0))$. Because $ u_{R}$ is radially decreasing, it follows 
	that $u_{R^{\prime}} \geq u_{R^{\prime}} (\Sigma(\cdot))$ in $U$. The parabolic comparison principle then yields
	$v_{R^{\prime}}(t,x) \geq w_{R^{\prime}}(t,x)$ for all $t>0$, $x\in U$.
	This being true for every line $H$, 
	we deduce that $v_{R^{\prime}}(t,\cdot)$ is radially decreasing.

	We derive the last statement of the lemma using the sliding method.
	Because the initial datum $u_{R}$ is a generalized stationary subsolution for the parabolic problem \eqref{bornesup}, 
	we have that $v_{R^{\prime}}(t,x)$ is increasing with respect to $t$ and converges 
	to a stationary solution $v_{\infty}>u_R$ of \eqref{bornesup} as $t$ goes to $+\infty$. Fix a direction
	$e\in \Sph$. 
	For $s \in [0,R'-R)$, define $
	u^s_R := u_{R}(\cdot -s e)$. We claim that
	$u^{s}_R < v_{\infty}$ in $B_R'$ for all $s\in[0,R'-R)$.
	Assume by contradiction that this is not the case, then call
	$s^{\star}$ the infimum of the $s\in[0,R'-R)$ for which $u^{s}_R < v_{\infty}$ fails.
	We deduce that there is a contact point between $v_{\infty}$ and $u_R^{s^\star}$, i.e.,
	$\min_{B_{R'}}(v_{\infty}-u_R^{s^\star})=0$.
	The  elliptic strong maximum principle then yields $v_{\infty} \equiv u_R^{s^\star}$, which is  
	impossible because $u_R^{s^\star}$ is compactly supported in $B_{R^{\prime}}$. 
	This proves the claim. We infer in particular that $v_{\infty}>M':=u_R(0)$ in the segment
	connecting $0$ (included) to $(R'-R)e$ (excluded). The desired property then follows from the arbitrariness of 
	$e\in \Sph$.
\end{proof}

These lemmas at hand, we can turn to the proof of 
Proposition~\ref{prop oriented invasion}.

\begin{proof}[Proof of Proposition \ref{prop oriented invasion}.] The proof is divided into three steps. 
	Let $u$ denote the solution of \eqref{evol homogeneous} set on the domain $\Omega_{3}$ given by \eqref{definition domain} emerging from the initial datum $u_{R}(\.-2Re_1)$. Because $2Re_1$ belongs to the set $\mc{S}$ of Lemma \ref{slidable}, this initial datum 
	is a generalized subsolution to \eqref{evol homogeneous}.
	As a consequence, $u(t,x)$ is increasing with respect to $t$ and that it converges as $t\to+\infty$ 
	locally uniformly in $x\in \ol\O_{3}$ to a stationary solution of \eqref{evol homogeneous}, that we call
	$u_{\infty}$. We claim that $u_\infty$ fulfills the following properties:
\begin{equation}\label{a montrer 1}
u_{\infty} > M \quad \text{ in }\  \ol\Omega_{3}\cap \mathcal{C},
\end{equation}
\begin{equation}\label{a montrer 2}
\forall z\in\{0,L_{1}e_{1},L_{2}e_{2},-L_{2}e_{2}\},\quad
u_{\infty}(\cdot+z) > u_{R}(\cdot-2Re_1).
\end{equation}

\medskip
\emph{Step 1. Estimate in the ``slidable'' region $\mc{S}$.} \\
We start with showing that
\begin{equation}\label{S+1}
u_{\infty}> M\quad \text{ in  }\  (\mc{S}+\ol B_1)\cap\ol\O_3,
\end{equation}
where $M$ is given in \eqref{inner estimate}. 
Consider an arbitrary $z$ in the set $\mc{S}$ defined in Lemma~\ref{slidable}.
Because $\mathcal{S}$ is 
path-connected, there exists a continuous path $\gamma: [0,1] \to \mathcal{S}$ such that 
$\gamma(0) =2R e_1$, $\gamma(1) = z$.
Let $(u^s_R)_{s\in[0,1]}$ be the continuous family of functions defined~by
$$
u^s_R := u_{R}(\cdot - \gamma(s)).
$$
We know from Lemma~\ref{slidable} that all these functions are generalized subsolutions 
to~\eqref{evol homogeneous}. Furthermore, $ u^0_R$ coincides with the initial datum of $u$, whence
$  u^0_R< u_\infty$. 
Then, the same sliding method as in the proof of Lemma~\ref{lem:vR} shows that
$ u^s_R < u_\infty$ for all $s\in[0,1]$, and thus in particular 
\Fi{u>uR}
\forall z\in\mc{S},\quad
u_{\infty}>u_R(\.-z).
\Ff
We eventually deduce from \eqref{inner estimate} that $\min_{\ol B_1(z)} u_{\infty}>\min_{\ol B_1} u_R= M$.
Because $z\in\mc{S}$ was arbitrary, \eqref{S+1} follows.


\medskip

\emph{Step 2. ``Jumping above" the narrow passage.} \\
In this step, we show that
\begin{equation}\label{B3R}
u_{\infty}> M^{\prime} \quad \text{ in }\ \O_3\cap B_{3R}(L_1e_1),
\end{equation}
where $M'$ is given by Lemma~\ref{lem:vR}.
By \eqref{u>uR} we know that $u_\infty>u_R(\.-L_1e_1)$.
Consider the function $v_{4R}$ provided by Lemma~\ref{lem:vR} with $R^{\prime} = 4R$. We extend it by $0$ outside
$B_{4R}$ and consider its restriction to $\O_3$. Let us show that $v_{4R}$ is a generalized subsolution for 
\eqref{evol homogeneous}. This property is trivial for the first equation.
For the second one, being $v_{4R}(t,\cdot)$ radially symmetric and decreasing thanks to Lemma~\ref{lem:vR},
we know that we need to check that condition \eqref{condition subsolution} holds with $z=0$ and $R$ replaced by~$4R$. 
This is readily achieved by noticing that $(\ol B_{4R}\cap\partial\O)\subset([0,4R]\times[-\frac{L_2}2,\frac{L_2}2])$ thanks to \eqref{h prime}, and then using the fact that $h$ is non-increasing on~$[0,4R]$. 
By periodicity, the function $v_{4R}(\.,\.-L_1e_1)$ is a subsolution to \eqref{evol homogeneous}
too. In addition, it is equal to $u_R(\.-L_1e_1)<u_\infty$ at time $t=0$.
It then follows from the comparison principle that $u_{\infty} > v_{4R}(t,\cdot-L_1e_1)$ for all $t>0$. 
Property~\eqref{B3R} eventually follows from the last statement of Lemma~\ref{lem:vR}.

\medskip

\emph{Step 3. Conclusion} \\
Gathering together \eqref{S+1} and \eqref{B3R} we obtain \eqref{a montrer 1}. 
Next, observe that the set
$$
\mathcal{S}\cup (\mathcal{S} +\{L_2 e_2\}) \cup (\mathcal{S} - \{L_2 e_2\}) 
$$
is path-connected and contains the points 
$$2Re_1,\qquad 
2Re_1+L_2 e_2 ,\qquad 2Re_1-L_2 e_2.$$
We can then argue as in the  step 1 to derive \eqref{u>uR} at those points $z$. The same conclusion holds
with $z=2Re_1+L_1 e_1$ by \eqref{B3R}.
This proves \eqref{a montrer 2}.

Now, because the convergence of $u(t,x)$ to $u_{\infty}$ is locally uniform in $x\in\ol\O_3$ as $t$ goes to $+\infty$, 
properties \eqref{a montrer 1}-\eqref{a montrer 2} hold true with $u_\infty$ replaced by $u(t,\.)$ for every 
$t$ larger than some $T$.
Owing to the comparison principle, it follows from the latter that 
$$\forall a\in\{0,1\},\ b\in\{-1,0,1\}\ \text{ such that }\
a+|b|\leq 1,\quad
u(\.+T,\cdot+aL_{1}e_{1}+bL_{2}e_{2}) > u.$$
Then, by iteration, for all $n\in\N$ there holds
$$\forall a\in\N\cup\{0\},\ b\in\Z\ \text{ such that }\
a+|b|\leq n,\quad
u(\.+nT,\cdot+aL_{1}e_{1}+bL_{2}e_{2}) > u.$$
The proof is thereby achieved because $u(t,x)>M$ for $t\geq T$, $x\in\ol\O_3\cap\mc{C}$.
\end{proof}

\subsubsection{Proof of Theorem \ref{th oriented}}\label{proof oriented invasion}

We first show the invasion property in the direction $e_{1}$, and then the blocking in the direction $-e_{1}$.

\begin{proof}[Proof of Theorem \ref{th oriented}.]
Let $\O_3$, $R$, $L_1$ and $L_2$ be as in the previous subsections.

\medskip
\emph{Statement $(i)$.} \\
Consider the solution  $u$ of Proposition \ref{prop oriented invasion}.
Let us show that Theorem~\ref{th oriented}$(i)$ holds for~$u$.
 First, because $u_R<1$, the parabolic comparison principle yields $u \leq 1$.
Now, call $c:=L_1/T$, with $T$ given by Proposition \ref{prop oriented invasion}. Take $c_1,c_2$ such that
	$0<c_1<c_2<c$. Consider a diverging sequence $(t_n)_{n\in \mathbb{N}}$ and
	a sequence $(x_n)_{n\in \mathbb{N}}$ in $\O_3$ such that
	$$\forall n\in\N,\quad c_1 t_n\leq x_n \cdot e_{1}\leq c_2 t_n,
	\qquad (x_{n}\cdot e_{2})_{n\in\N}\text{ is bounded}.$$
	Consider then the sequence $(k_n)_{n\in \mathbb{N}}$ in $\Z$ for which
	$0\leq x_n \cdot e_{1}-k_nL_1<L_1$.
	This sequence diverges to $+\infty$ because $(x_n\.e_1)_{n\in \mathbb{N}}$ does. 
	We define the translated functions	
$$
u^{n} := u(\cdot+t_n,\cdot+k_n e_{1} L_1).
$$	
	The sequence $(u^{n})_{n\in \mathbb{N}}$ converges as $n$ goes to $+\infty$ (up to extraction)
	locally uniformly in $t\in\R$, $x\in\ol\O_3$,
	to a function
	$ u^{\infty}$ which is an entire solution of \eqref{evol homogeneous}.
We claim~that
\begin{equation}\label{u infty}
\forall t\in \mathbb{R}, \ \forall x \in \O_3 , \quad  u^{\infty}(t,x) \geq M,
\end{equation}
where $M$ is given by \eqref{inner estimate}. Fix $t\in \mathbb{R}$ and $x\in \O_3$.
Set $m_{n} := \left\lfloor{\frac{t+t_{n}}{T}}\right\rfloor$-1, where $\lfloor \cdot \rfloor$ stands for the integer part. 
We compute
$$k_n\leq \frac{x_n\.e_1}{L_1}\leq \frac{c_2}{L_1} t_n<
\frac{c_2}{L_1}
(m_n+2)T-\frac{c_2}{L_1}t=
(m_n+2)\frac{c_2}{c}-\frac{c_2}{L_1}t.
$$
Because $c_2<c$, we find that $m_{n} - k_{n} \to +\infty$ as $n$ goes to $+\infty$.
Consequently, for $n$ large enough, there holds
$$ t+t_{n} \geq (m_{n}+1)T \ \text{ and }\  x+ k_n L_1 e_1 \in  
\bigcup_{\substack{a\in\N\cup\{0\},\ b\in\Z  \\ 
		a+|b|\leq m_n}} \left( \mathcal{C}+\{aL_{1}e_{1}+bL_{2}e_{2}\} \right),$$
and therefore $u^n(t,x)>M$ thanks to Proposition \ref{prop oriented invasion}.
This proves \eqref{u infty}. Recall that $M>\theta$. Then Lemma \ref{ODE}
yields $u^\infty\equiv1$. This means that property $(i)$ of Theorem~\ref{th oriented}
holds for the solution $u$.
As usual, one then extend the result to the class of initial data stated in the theorem by 
arguing as in the proof of Theorem~\ref{th persistence}.

\medskip
\emph{Statement $(ii)$.} \\ 
We shall make use of the blocking property for cylinders. Consider the cylinder
$$\tilde{\Omega}_{3} :=  \O_3\cap(\R\times (0 , L_{2})).$$
Because of the choice of $\varepsilon$ in the construction of $\Omega_{3}$, we can apply Proposition \ref{prop bbc} to this cylinder with $b = -1$ and get a positive solution $w$ to \eqref{eq block L}.
We first extend $w$ to the whole $\t\O_3$ by setting $w(x) = 1$ for $x\cdot e_{1} \geq 1$. Next, we extend it by periodicity 
in the direction $e_2$, with period $L_2$. Because $\nabla w(x)\.e_2=0$ for $x\.e_2\in L_2\Z$,
we have that $w$ is a generalized supersolution to~\eqref{evol homogeneous}.
Consider now a compactly supported initial datum $u_{0}\leq1$.
We can find $k\in \mathbb{Z}$ such that $u_{0} \leq w(\cdot + k L_1 e_{1} )$. 
The parabolic comparison principle then yields $u(t,x)\leq w(x+2L_1k e_1)$, for all $t>0, \ x \in \Omega_{3}$. 
Statement $(ii)$ eventually follows from the last property of~\eqref{eq block L}.		
\end{proof}






\end{document}